\documentclass[preprint,11pt]{elsarticle}

\usepackage{graphicx}
\usepackage{amssymb}
\usepackage{amsmath}
\usepackage{amsthm}
\usepackage{latexsym}
\usepackage{eufrak}
\usepackage{euscript}
\usepackage{exscale}
\usepackage{graphicx}
\usepackage{mathrsfs}
\usepackage{lineno}
\usepackage{hyperref}
\usepackage[table,xcdraw]{xcolor}
\newcommand{\rb}[1]{ \left( #1 \right) }
\newcommand{\rrb}[1]{ \left\lbrace #1 \right\rbrace }
\newcommand{\curbr}[1]{ \lbrace #1 \rbrace }
\newcommand{\cb}[1]{ \left[ #1 \right] }
\newcommand{\abs}[1]{\left|#1\right|}

\newcommand{\norm}[1]{\left\lVert#1\right\rVert}
\newcommand{\normalt}[1]{\lVert#1\rVert}
\newcommand{\suml}{\sum\limits}
\newcommand{\intl}{\int\limits}

\newcommand{\maxl}{\max\limits}

\newcommand{\Exp}[1]{ \mathbb{E}\left[ #1 \right] }

\newcommand{\CExp}[2]{ \mathbb{E}\left[ #1 | #2 \right] }

\newcommand{\kc}{\mathcal{C}}

\newcommand{\Mu}{\operatorname{M}}

\newcommand{\de}{\mathrm{\,d}}

\newcommand{\CCal}{\mathcal{C}}

\newcommand{\E}{\mathbb{E}}
\newcommand{\Ibb}{\mathbb{I}}
\newcommand{\N}{\mathbb{N}}

\newcommand{\R}{\mathbb{R}}

\usepackage{dsfont}
\newcommand{\ONE}{\mathds{1}}

\newcommand{\xvec}{\mathbf{x}}
\newcommand{\Xvec}{\mathbf{X}}

\journal{JMVA}

\begin{document}

\begin{frontmatter}

\title{Mean and quantile regression in the copula setting: \\ properties, sharp bounds and a note on estimation}

\author[tru]{Henrik Kaiser}
\ead{henrik.kaiser@plus.ac.at}
\author[tru]{Wolfgang Trutschnig\corref{cor1}}
\ead{wolfgang@trutschnig.net}
\cortext[cor1]{Corresponding author}

%\address[tru]{European Centre for Soft Computing, Edificio Cient\'ifico Tecnol\'ogico \\
%Calle Gonzalo Guti\'errez Quir\'os, s/n,  $3^{a}$ planta, 33600 Mieres (Asturias), Spain\\
% Tel.: +34 985456545  \\
% Email: wolfgang.trutschnig@softcomputing.es
%}

%\ead{ wolfgang.trutschnig@softcomputing.es}
%\address[henrik]{Department for Artificial Intelligence and Human Interfaces  
%University Salzburg, Hellbrunnerstrasse 34, A-5020 Salzburg, Austria
%}

\address[tru]{Department for Artificial Intelligence and Human Interfaces,   
University Salzburg, Hellbrunnerstrasse 34, A-5020 Salzburg, Austria\\ Tel: +43 662 8044 5326
Fax: +43 662 8044 137
}

\begin{abstract}
Driven by the interest on how uniformity of marginal distributions 
propa\-gates to properties of regression functions, in this contribution we tackle the following 
questions: Given a $(d-1)$-dimensional random vector $\Xvec$ and a
random variable $Y$ such that all univariate marginals of $(\Xvec,Y)$ are uniformly 
distributed on $[0,1]$, how large can the average absolute deviation of the mean 
and the quantile regression function of $Y$ given $\Xvec$ from the value $\frac{1}{2}$ be, and how much 
mass may sets with large deviation have? 
We answer these questions by deriving sharp inequalities, both in the mean as well as in the 
quantile setting, and sketch some cautionary consequences to nowadays quite 
popular pair copula constructions involving the so-called simplifying assumption. 
Rounding off our results, 
working with the so-called empirical checkerboard estimator in the bivariate setting, we
show strong consistency for both regression types and illustrate the speed of 
convergence in terms of a simulation study. 
\end{abstract}

\begin{keyword}
%% keywords here, in the form: keyword \sep keyword
copulas \sep regression \sep quantiles \sep estimation \\[0.4cm]
%% MSC codes here, in the form: \MSC code \sep code
%% or \MSC[2008] code \sep code (2000 is the default)
 \MSC[2020] 60E05 \sep 62G08 \sep 62H20 \sep 60E15
\end{keyword}

\end{frontmatter}

%%
%% Start line numbering here if you want
%%
% \linenumbers

\newtheorem{thm}{Theorem}
\newtheorem{lem}[thm]{Lemma}
\newtheorem{prop}[thm]{Proposition}
\newtheorem{cor}[thm]{Corollary}
\newdefinition{rmk}[thm]{Remark}
\newdefinition{ex}[thm]{Example}
\newdefinition{defin}[thm]{Definition}
%\newproof{pot}{Proof of Theorem \ref{thm2}}

\section{Introduction}

Regression methods, particularly mean and quantile regression, play a fundamental role 
throughout all quantitative fields. Traditionally, the focus lies on estimating the 
conditional mean of a response variable $Y$, given the values of an ensemble of 
explanatory variables $\Xvec$, with the aim to summarize the relationship between covariates
and outcome or to predict the value of $Y$, based on new observations of $\Xvec$. 
Trying to capture more information on the distribution of the 
outcome $Y$ given an observation of $\Xvec$, 
quantile regression (see, e.g., \cite{Koenker2005,KoenkerBassett1978})
provides a more comprehensive understanding of the response distribution across its entire range.

According to Sklar's famous theorem (see, \cite{DS2016,Nelsen2006}, copulas constitute 
the link between (continuous) multivariate distributions functions and their 
univariate marginals - as such, they capture all scale-invariant dependence between 
random variables. Motivated by this fact, we here study mean and quantile 
regression in the context of 
$d$-dimensional copulas, interpreting the first $d-1$ coordinates as covariates and 
the $d$-th coordinate as outcome. Our focus is not on separate estimation 
of the marginals and the underlying co\-pulas as done in \cite{DetteETAL2014} - inspired 
by curiosity about how uniformity of the marginals translates/propagates to characteristics 
of regression functions, the goal of this contribution is twofold: 
firstly, to derive best-possible upper bounds for the maximal $L^p$-deviation of the 
regression function from the value $\frac{1}{2}$ corresponding to the regression function describing independence. And, secondly, to determine best-possible bounds for the 
mass of sets with large deviation.  
In other words, assuming that $\Xvec$ is $(d-1)$-dimensional random vector and
$Y$ is a random variable on a joint probability space $(\Omega,\mathcal{A},\mathbb{P})$, 
such that $(\Xvec,Y)$ has copula $C$ as distribution function
(restricted to $[0,1]^d$), we want to know, how large 
$$
\int_{\Ibb^{d-1}} \vert \mathbb{E}(Y | \Xvec=\xvec) - \tfrac{1}{2}\vert^p \de \mathbb{P}^\Xvec(\xvec), \hspace{1cm} p \in [1,\infty),
$$
as well as 
$$
\mathbb{P}^\Xvec\rb{ \rrb{ \xvec \in \Ibb^{d-1}: 
\abs{ \mathbb{E}(Y | \Xvec=x) - \tfrac{1}{2} } \geq a } }
$$
can possibly be. After providing sharp inequalities for mean regression we ask (and answer) 
the analogous questions for quantile regression. 
To the best of our knowledge, these natural properties  
have not been investigated in the literature yet. 
While our original intention was to provide answers to the afore-mentioned question 
primarily for the family of so-called linkages, modeling the situation in which the 
covariates are independent (which is particularly useful for constructing
multivariate dependence measures (see \cite{Griessenberger2022}), 
we decided to state and prove our results for the general setting, since very little additional
technical effort is required. 

Throughout the past decade pair copula constructions in combination with the so-called 
simplifying assumption (praised for their flexibility also in the high-dimensional setting, see \cite{AasCzadoBakken2009, Joe2014,NAGLER}) have become more and more 
popular. As shown in \cite{MrozEtAl2021} (also see the very recent survey \cite{NAGLER2025}),
these constructions may suffer from very poor approximation quality 
and should therefore be handled with care. 
Our results on regression underline the fact that the afore-mentioned warning extends to 
the regression setting. 

%These tools have found wide applications, from economics and health sciences to finance and insurance \cite{FuchsTrutschnig2020}. Often, however, classical regression assumes simple or linear dependencies between variables, which may not be able to capture the complexity inherent in a real-world dataset. Conversely, copula theory offers a flexible framework to model more sophisticated dependencies. In particular, thereby, a more accurate representation of the relationships between the involved variables is possible, even in the presence of nonlinear or asymmetric dependencies. In copula-baded regression settings, it is reasonable to confine the study to copula families, in which the joint copula of the covariates is fixed. Of special importantance in this regard is the linkage class, where all covariates are assumed independent, which is the most desirable property in practice \cite{Griessenberger2022}. Copula architectures can become quite complicated in higher dimensions. An attempt to fix this issue are so-called vine copulas and the simplifying assumption \cite{AasCzadoBakken2009, Joe2014}. This approach, however, bears a certain risk of misspecification, as the simplifying assumption may not be true. In fact, obtainable results may differ a lot from the target \cite{MrozEtAl2021}.

The remainder of this paper is organized as follows: After gathering notation and 
preliminaries in Section 2, we study the absolute deviation of the mean regression function from its mean $\frac{1}{2}$ and answer the afore-mentioned two questions by 
providing best-possible bounds. 
Section 4 focuses on quantile regression and again establishes sharp bounds. Finally, in 
Section 5, we study the bivariate setting and prove strong consisteny of the empirical 
checkerboard estimator for the mean and the quantile regression function, 
without any regularity conditions for the copula $C$. A small simulation study, 
illustrating the obtained convergence results, rounds off the paper.

%--------------------------------------------------------------------------------------------------------------
\section{Notation and Preliminaries} \label{SecNotPre}

For every metric space $(S,d)$ we will let $\mathcal{B}(S)$ denote the Borel $\sigma$-field on $S$.
Throughout this article, $\Ibb := [0, 1]$ denotes the closed unit interval, the dimension is denoted by $d \in \mathbb{N} \setminus \curbr{1}$, and bold symbols refer to vectors, e.g., $\xvec := (x_1,x_2,\ldots, x_{d}) \in \mathbb{R}^{d}$. For $\xvec := (x_1,x_2,\ldots, x_{d}) \in \mathbb{R}^{d}$ and $l \in \{1,\ldots,d\}$\textcolor{blue}{,} we will 
let $\xvec_{1:l}$ denote the vector $(x_1,\ldots,x_l) \in \mathbb{R}^l$. 
Since the main focus of this contribution is regression\textcolor{blue}{,} we will in particular write 
$(\xvec, y) = (x_1, \ldots, x_{d-1}, y) \in \mathbb{R}^d$ for all $\xvec \in \mathbb{R}^{d-1}$ 
and $y \in \mathbb{R}$.  
Furthermore, $\lambda_d$ stands for the $d$-dimensional Lebesgue-measure on $\mathcal{B}(\mathbb{R}^d)$ 
or $\mathcal{B}(\Ibb^d)$, with $\lambda := \lambda_1$, for brevity. 
The family of $d$-dimensional copulas is indicated by $\mathcal{C}^d$, the uniform metric on $\mathcal{C}^d$ is defined by
$$d_\infty(A,B):= \max_{(\xvec, y) \in \Ibb^d} \left| A(\xvec, y) - B(\xvec, y) \right| \hspace{1cm} (A, B \in \CCal^d).$$
It is well known that $(\mathcal{C}^d, d_\infty)$ is a compact metric space (see 
\cite{DS2016, Nelsen2006}). Specifically $\Pi_d$ and $\Mu_d$ denote the $d$-dimensional product (or independence) and minimum copula, respectively, i.e., $\Pi(\xvec, y) =y \prod_{j=1}^{d-1} x_j$ and $\Mu_d(\xvec, y) = \min\{x_1, \ldots, x_{d-1}, y\}$. In the bivariate case, we simply write $\Pi := \Pi_2$ and $\Mu := \Mu_2$.  \\
Given $C \in \mathcal{C}^d$, the corresponding $d$-stochastic measure is denoted by $\mu_C$, i.e., $\mu_C([\mathbf{0},\xvec]\times [0, y]) := C(\xvec, y)$ for all $(\xvec, y) \in \Ibb^d$, where $[\mathbf{0}, \xvec] := \times_{i=1}^{d-1} [0, x_i]$. 
For every vector $\mathbf{j}=(j_1,\ldots,j_l) \in \{1,\ldots,d\}^l$\textcolor{blue}{,} with $j_1<j_2<\ldots<j_l$\textcolor{blue}{,} we will let 
$C_\mathbf{j}$ denote the marginal copula of $C$ with respect to the coordinates $\mathbf{j}$. 
In other words, defining the projection $\pi_\mathbf{j}: \Ibb^d \rightarrow \Ibb^l$ by
$$
\pi_\mathbf{j}(x_1,\ldots,x_d)=(x_{j_1},x_{j_2},\ldots,x_{j_l})\textcolor{blue}{,}
$$
we have $C_\mathbf{j}$ corresponds to the push-forward $\mu_C^{\pi_\mathbf{j}}$ of $\mu_C$ via $\pi_\mathbf{j}$.  
To keep notation simple (and in accordance with $x_{1:l}$ from above), in the case of 
$\mathbf{j}=(1,2,\ldots,d-1)$\textcolor{blue}{,} we will simple write 
$C_{\mathbf{j}}=C_{1:(d-1)} \in \CCal^{d-1}$ and refer to $C_{1:(d-1)}$ as marginal copula of $C$ 
with respect to the first $d-1$ coordinates; analogously, in the case of $\mathbf{j}=(1,2,\ldots,d-2,d)$\textcolor{blue}{,}
we will write $C_{\mathbf{j}}=C_{1:(d-2),d} \in \CCal^{d-1}$.
Finally, the $d$-flipped  (flipped with respect to the last coordinate) of $C \in \CCal^d$ is defined by $\overline C(\xvec, y) := C_{1:(d-1)}(\xvec) - C(\xvec, 1-y)$, obviously we have $\overline C \in \CCal^d$.

Besides $d$-stochastic measures, in what follows (regular) conditional distributions of a copula will be of 
special importance. 
Suppose that $(\Omega_1,\mathcal{A}_1)$ and $(\Omega_2, \mathcal{A}_2)$ be measurable spaces. Then\textcolor{blue}{,} a map $K: \Omega_1 \times \mathcal{A}_2 \to \Ibb$ is called Markov kernel (a.k.a. transition probability) 
from $(\Omega_1,\mathcal{A}_1)$ to $(\Omega_2,\mathcal{A}_2)$, if the map $\omega_1 \mapsto K(\omega_1, A_2)$ is $\mathcal{A}_1$-$\mathcal{B}(\mathbb{R})$-measurable, for every fixed $A_2 \in \mathcal{A}_2$, and the map $A_2 \mapsto K(\omega_1, A_2)$ is a probability measure on $\mathcal{A}_2$,  for every fixed $\omega_1 \in \Omega_1$.  
Considering a random variable $Y$ and a $(d-1)$-dimensional random vector $\Xvec$ on a joint 
probability space $(\Omega, \mathcal{A}, \mathbb{P})$, 
%i.e., measurable mappings  $Y:(\Omega, \mathcal{A}, \mathbb{P}) \to (\mathbb{R},\mathcal{B}(\mathbb{R}))$ and $\Xvec:(\Omega, \mathcal{A}, \mathbb{P}) \to (\mathbb{R}^{d-1},\mathcal{B}(\mathbb{R}^{d-1}))$, 
a Markov kernel $K: \mathbb{R}^{d-1} \times \mathcal{B}(\mathbb{R}) \rightarrow \Ibb$ is said to be a regular conditional distribution of $Y$ given $\Xvec$, 
if for every $F \in \mathcal{B}(\mathbb{R})$ the following assertion holds: for $\mathbb{P}$-almost every $\omega \in \Omega$ we have
\begin{align*}
	K(\Xvec(\omega), F) = \mathbb{E}(\ONE_F \circ Y | \Xvec)(\omega). 
\end{align*}
For each random vector $(\Xvec, Y)$, it is well-known that a regular conditional distribution $K$ of 
$Y$ given $\Xvec$ exists and is unique for $\mathbb{P}^{\Xvec}$-a.e. $\xvec \in \mathbb{R}^{d-1}$, with $\mathbb{P}^{\Xvec}$ denoting the push-forward of $\mathbb{P}$ under $\Xvec$. It is also well known that $K$ only depends on $\mathbb{P}^{(\Xvec,Y)}$. We will write $(\Xvec,Y) \sim C$ 
if the copula $C$ is the distribution function of $(\Xvec,Y)$ restricted to $\Ibb^d$. 
Finally, without loss of generality, we interpret 
the Markov kernel $K_C$ of $C \in \CCal^d$ (with respect to the first $(d-1)$-coordinates) as a mapping $K_{C}:\Ibb^{d-1} \times \mathcal{B}(\Ibb) \to \Ibb$.
For further background on copulas, $d$-stochastic 
measures, conditional expectation and Markov kernels we refer to \cite[][]{DS2016, Klenke2020, Nelsen2006, Kallenberg2002}.

Generally speaking, disintegration theorems refer to integral representations of multivariate measures in terms of marginals and conditional distributions. 
In the case of $d\geq 3$ many such representations exist 
(see \cite{Kallenberg2002}). In its simplest form, in the copula setting we have 
\begin{align}\label{eq:disint}
	\mu_C(\mathbf{B} \times F) = \int_{\mathbf{B}} K_{C}(\xvec,F) \, \de\mu_{C_{1:(d-1)}}(\xvec)\textcolor{blue}{,}
\end{align}
for every $\mathbf{B} \in \mathcal{B}(\Ibb^{d-1})$ and every $F \in \mathcal{B}(\Ibb)$, implying in particular 
that the $d$-th univariate maginal is uniform on $\Ibb$. 
In the sequel\textcolor{blue}{,} we will especially use the following property on the relationship between projections and
Markov kernels, the proof of which is provided in \ref{AppProNotPrel}.

\begin{lem} \label{Lem2025062701}
Suppose that $d \geq 3$, that $C \in \CCal^d$ and let $F \in \mathcal{B}(\Ibb)$ be arbitrary but fixed. 
Then\textcolor{blue}{,} for $\lambda_{d-2}$-a.e. $\xvec_{1:d-2} \in \Ibb^{d-2}$\textcolor{blue}{,} we have
\begin{align} \label{2025070202}
\intl_{ \Ibb } K_C( \xvec_{1:d-1}, F ) K_{C_{1:(d-1)}}(\xvec_{1:d-2}, \de x_{d-1}) = 
K_{C_{1:(d-2), d}}(\xvec_{1:d-2}, F ).
\end{align}
%\begin{enumerate}
%\item For $\lambda_{d-1}$-a.e. $\xvec_{d-1} \in \Ibb^{d-1}$,
%\begin{align} \label{2025070201}
%\mu_{ C_{1:(d-1)} }([\mathbf{0}, \xvec_{d-1}]) = \intl_{ [\mathbf{0}, \xvec_{d-2}] } K_{ C_{1:(d-1)} }( \mathbf{u}, %[0, x_{d-1}] ) \de \mu_{ C_{1:(d-2)} }(\mathbf{u}),
%\end{align}
%i.e., $\mu_{ C_{1:(d-1)} }$ is absolutely continuous with respect to $\mu_{ C_{1:(d-2)} }$, with density $K_{ C_{1:(d-1)} }$.
%\item For $\lambda_{d-2}$-a.e. $\xvec_{d-2} \in \Ibb^{d-2}$ and $B \in \mathcal{B}(\Ibb)$,
%\begin{align} \label{2025070202}
%\intl_{ \Ibb } K_C( \xvec_{d-1}, B ) K_{C_{1:(d-1)}}(\xvec_{d-2}, \de x_{d-1}) = K_{C_{1:(d-2), d}}(\xvec_{d-2}, B ).
%\end{align}
%\end{enumerate}
\end{lem}
The (mean) regression function $r_C$ of a copula $C \in \CCal^d$ (with respect to the first $d-1$ coordinates), i.e., the function $\xvec \mapsto \E(Y | \Xvec=\xvec)$ for $(\Xvec,Y) \sim C$, can be expressed in 
terms of the Markov kernel $K_C$ as 
\begin{align} \label{2025051301}
r_C(\xvec) := \intl_\Ibb y K_C(\xvec, \de y) \hspace{1cm} (\xvec \in \Ibb^{d-1}),
\end{align}
or, equivalently as 
\begin{align} \label{2025051303}
r_C(\xvec) = \intl_{ \Ibb } K_C(\xvec, (y, 1] ) \de y \hspace{1cm} (\xvec \in \Ibb^{d-1}).
\end{align}
Obviously, for the flipped copula $\overline C$\textcolor{blue}{,} we have $r_{ \overline C } = 1 - r_C$.

\noindent For fixed $C \in \CCal^d$, $\xvec \in \Ibb^{d-1}$ and $\tau \in [0,1]$, the $\tau$-quantile 
function $Q_C^\tau$ is defined by 
\begin{align}\label{eq:def.quantile}
Q_C^\tau(\xvec) := \inf\curbr{ y \in \Ibb : K_C(\xvec, [0, y]) \geq \tau } = (F^\xvec_C)^-(\tau), 
\end{align}
where $(F^\xvec_C)^-$ denotes the quasi-inverse of the conditional distribution 
function $y \mapsto F^\xvec_C(y)=K_C(\xvec, [0, y]) $. 
In the sequel, we will only consider $\tau \in (0,1]$ since, by definition, for $\tau=0$ we have
$Q_C^\tau(\xvec)=0$ for every $C \in \CCal^d$ and every $\xvec \in \Ibb^{d-1}$. 
It is well known (and straightforward to verify) that $y_0 < Q_C^\tau( \xvec)$ if, and only 
if $K_C(\xvec, [0, y_0]) < \tau$. Moreover, $Q_C^\tau( \xvec) < y_0$ implies $K_C(\xvec, [0, y_0]) \geq \tau$. 
As a direct consequence, for all $(y_0, \tau) \in \Ibb^2$\textcolor{blue}{,} we have  
$$
\curbr{ \xvec \in \Ibb^{d-1} : Q_C^\tau(\xvec) \leq y_0 } = \curbr{ \xvec \in \Ibb^{d-1} : K_C(\xvec, [0, y_0]) \geq \tau }\textcolor{blue}{,}
$$
so that measurability of $\xvec \mapsto Q_C^\tau(\xvec)$ directly follows from measurability of 
$\xvec \mapsto K_C(\xvec,F)$ for every $F \in \mathcal{B}(\Ibb)$.

A handy and frequently used class of copulas are so-called 
checkerboard copulas (see \cite{Miku,Griessenberger2022,Trutschnig2017} and the references therein). 
Roughly speaking, these copulas are characterized by the property that they locally, on hypercubes of equal volume, resemble shrinked copies of $d$-dimensional copulas.
For $N \in \mathbb{N}$ and $i \in \{1,\ldots,n\}$ write $I_{N, i} := (\frac{i-1}{N}, \frac{i}{N})$. Then, the checkerboard $\Pi$-approximation of $C \in \CCal^d$ with resolution $N \in \N$, denoted by 
$\mathfrak{Cb}_N(C)$, is the copula with density (see \cite{Griessenberger2022} and the references therein)
\begin{align*}
\mathfrak{cb}_N(C)(\xvec, y) := N^d \sum_{i_1, \hdots, i_d=1}^N \mu_C( \times_{k=1}^d I_{N, i_k}) \ONE_{ \times_{k=1}^d I_{N, i_k} }(\xvec, y) \hspace{1cm} ( (\xvec, y) \in \Ibb^d ).
\end{align*}
By construction, the conditional distribution functions $y \mapsto K_{\mathfrak{Cb}_N(C)}(\xvec,[0,y])$ 
are piecewise linear and constant (in $\xvec$) on every hypercube $I_{N, i_1}\times \ldots \times I_{N, i_{d-1}}$.
As a direct consequence, the regression function $r_{\mathfrak{Cb}_N(C)}$ as well as each quantile 
regression function
$Q^\tau_{\mathfrak{Cb}_N(C)}$ is piecewise constant. 
%In particular, defining
%\begin{align*}
%    j_{i_1, \hdots, i_{d-1}}^*(\tau) := \min\rrb{ 1 \leq j \leq N : N^{d-1} \suml_{k=1}^j \mu_{ E_n }(I_{N, i_1} \times \hdots \times I_{N, i_{d-1}} \times I_{N, k}) \geq \tau },
%\end{align*}
%it is clear that
%\begin{align*}
%    Q_{\mathfrak{Cb}_N(C)}^\tau(\xvec) = \frac{1}{N} \suml_{i_1, \hdots, i_{d-1}=1}^N j_{i_1, \hdots, i_{d-1} }^*%(\tau) \ONE_{ I_{N, i_1} \times \hdots \times I_{N, i_{d-1}} }(\xvec) \hspace{1cm} ((\xvec, \tau) \in \Ibb^d).
%\end{align*}
Specifically, if $E_n$ denotes the empirical copula (i.e., the multilinear interpolation of the subcopula induced by the pseudo-ranks) 
of a sample $(\Xvec_1, Y_1), \hdots, (\Xvec_n, Y_n)$ from 
$(\Xvec,Y) \sim C$, we refer to $\mathfrak{Cb}_N(E_n)$ as the empirical checkerboard 
approximation with resolution $N=N(n)$ of $E_n$, or, shortly as empirical $N$-checkerboard. 

Our subsequent discussion is mostly focused on copula families with fixed $1:(d-1)$-marginal 
$A \in \CCal^{d-1}$. 
Formally, for $d\geq 3$ and a fixed $A \in \CCal^{d-1}$, we will consider families of the form
\begin{align*}
\CCal_A^d := \rrb{ C \in \CCal^d : C_{1:(d-1)}(\xvec) = A(\xvec) \mbox{ for all } \xvec \in \Ibb^{d-1} },
\end{align*}
and refer to $A$ as the copula of the covariates. Considering that univariate marginals of 
copulas coincide with $\lambda$ on $\Ibb$, we will formulate all results concerning $\CCal_A^d$ for
arbitrary $d\geq 2$ and interpret the case $d=2$ accordingly. 
For fixed $A \in \CCal^{d-1}$ define the $L^p$-norm of a measurable function $f : \Ibb^{d-1} \rightarrow \R$ as
(notice the dependence on $A$ through the integrating measure)
\begin{align*}
\norm{ f }_{A, p} := \rrb{ \intl_{ \Ibb^{d-1} } \abs{ f(\xvec) }^p \de \mu_A(\xvec) }^\frac{1}{p} \hspace{1cm} (p \in [1, \infty)).
\end{align*}
The family  
\begin{align*}
\CCal_\Pi^d = \rrb{ C \in \CCal^d : C_{1:(d-1)}(\xvec) = \prod_{i=1}^{d-1} x_i \mbox{ for all } \xvec \in \Ibb^{d-1} }
\end{align*}
is known as the linkage class and has been used in \cite{Griessenberger2022} to construct a measure 
quantifying the extent of dependence of a random variable $Y$ on a random vector $\Xvec$ (which was
a direct extension of the original bivariate approach in \cite{JGT}). 
For $(\Xvec,Y) \sim C \in \CCal^d_\Pi$ obviously all covariates are independent, i.e., $\mu_{C_{1:(d-1)}} = \lambda_{d-1}$ holds, which is why we pay special attention to $\CCal^d_\Pi$. 
Following \cite{Griessenberger2022}, setting 
\begin{align} \label{2025062702}
\Phi_{C_1, C_2;p}(y) := \intl_{ \Ibb^{d-1} } \abs{ K_{C_1}(\xvec, [0, y]) - K_{C_2}(\xvec, [0, y]) }^p \de \lambda_{d-1}(\xvec)\textcolor{blue}{,}
\end{align}
for every $y \in \Ibb$, the metric $D_p$ on $\CCal_\Pi^d$ is defined by
\begin{align} \label{2025062701}
D_p(C_1, C_2) := \rrb{ \intl_{ \Ibb } \Phi_{C_1, C_2;p}(y) \de \lambda(y) }^\frac{1}{p} \hspace{1cm} (p \in [1, \infty)).
\end{align}
In the next section, we will extend this metric to families $\CCal_A^d$ and study its 
interrelation with the $L^p$-norm of regression functions. 

For deriving sharp inequalities for the maximal deviation of $r_C$ from the mean $\frac{1}{2}$\textcolor{blue}{,} we will 
work with the Hardy-Littlewood-Pólya theorem (see \cite[Ch. 1, Theorem D.2]{Marshall2011}) 
involving rearrangements. We therefore complete 
this section with some definitions concerning decreasing rearrangements and a simple lemma, which we will use in the sequel. For an arbitrary measurable function $f : \Ibb^{d-1} \rightarrow \mathbb{R}$ and 
$a \in \mathbb{R}$, we define the $a$-superlevel set $\cb{ f }_a $ and the strict $a$-superlevel set $\langle f \rangle_a $ by
\begin{align*}
\cb{ f }_a  := \rrb{ \xvec \in \Ibb^{d-1} : f(\xvec) \geq a } 
\end{align*}
and the strict $a$-superlevel set by
\begin{align*}
\langle f \rangle_a  := \rrb{ \xvec \in \Ibb^{d-1} : f(\xvec) > a }.
\end{align*}
For the remainder of this section let $C \in \CCal^d$ be arbitrary but fixed. 
%For an extension to $d \in \N \setminus \curbr{1, 2}$, assuming that a measurable $f_C : \Ibb^{d-1} \rightarrow \Ibb$ also depends on the Markov kernel $K_C$ (with respect to the first $d-1$ coordinates) of a copula $C \in \CCal^d$, we 
Defining the functions $m_{f,C}, \overline m_{f,C} : \mathbb{R} \rightarrow \Ibb$ by 
\begin{align}\label{def:survival.f}
m_{ f,C}(v) := \mu_{ C_{1:(d-1)} }( \langle f\rangle_v ), \nonumber\\
\overline m_{f,C}(v) := \mu_{ C_{1:(d-1)} }( \cb{ f}_v ),
\end{align}
obviously $m_{f,C}$ and $\overline m_{f,C}$ are non-increasing. Based on these two functions,
the so-called decreasing rearrangements 
$f_{C, \downarrow}, \overline f_{C, \downarrow} : \Ibb \rightarrow \mathbb{R}$ of $f$ are given by
\begin{align*}
f_{C, \downarrow}(u) :=&\, \sup\rrb{ v \in \Ibb : m_{f,C}(v) > u }, \\
\overline f_{C, \downarrow}(u) :=&\, \sup\rrb{ v \in \Ibb : \overline m_{f,C}(v) > u }.
\end{align*}
Obviously $f_{C, \downarrow}(u)$ and $\overline f_{C, \downarrow}(u)$ are both non-increasing functions too. 
To the best of our knowledge, available literature is confined to the decreasing rearrangement 
$f_{C, \downarrow}$ in the situation $\mu_{ C_{1:(d-1)} } = \lambda_{d-1}$. 
Since in the sequel we will work with non-strict superlevel sets, the following lemma, stating 
that there is no difference between these two versions, is useful (the proof can be found in \ref{AppProNotPrel}).

\begin{lem} \label{Lem20250922}
For each $d \geq 2$, $C \in \CCal^d$ and every measurable $f: \Ibb^{d-1} \rightarrow \Ibb$\textcolor{blue}{,} the two 
functions $f_{C, \downarrow}$ and $\overline f_{C, \downarrow}$ are identical. 
%\begin{align*}
%f_{C, \downarrow}(u) = \overline f_{C, \downarrow}(u) \hspace{1cm} (u \in \Ibb).
%\end{align*}
%In particular, $f_\downarrow = \overline f_\downarrow$, for any measurable $f : \Ibb \rightarrow \Ibb$.
\end{lem}
%\clearpage

\section{Mean Regression}

We now focus on mean regression and start with the following simple observation, stating 
that, as a direct consequence of working with $d$-stochastic measures, the regression 
function $r_C$ of $C$ integrates to a constant not depending on the copula $C$. 

\begin{lem}
For every $d \geq 2$ and every $C \in \CCal^d$ we have
\begin{align*}
\intl_{ \Ibb^{d-1} } r_C(\xvec) \de \mu_{ C_{1:(d-1)} }(\xvec) = \tfrac{1}{2} \equiv r_\Pi.
\end{align*}
\end{lem}

\begin{proof}
Using eq. (\ref{2025051303}), Fubini's theorem and disintegration, for any $C \in \CCal^d$, we directly get
\begin{align*}
\intl_{ \Ibb^{d-1} } r_C(\xvec) \de \mu_{ C_{1:(d-1)} }(\xvec) &= \intl_{ \Ibb^{d-1} } \intl_{\Ibb} K_C(\xvec, (y, 1] ) \de \lambda(y) \de \mu_{ C_{1:(d-1)} }(\xvec) \\
&= \intl_{\Ibb} (1-y) \de \lambda(y) = \tfrac{1}{2}.
\end{align*}
The fact that $r_\Pi(\xvec) = \frac{1}{2}$ for every $\xvec \in \Ibb^{d-1}$ is trivial. 
\end{proof}

In the sequel, given a fixed copula family $\CCal_A^d$, we consider $\normalt{r_C-\frac{1}{2}}_{A, p}$ as a measure for the average deviation of $r_C$ from its mean, or equivalently, from the regression 
function of the copula corresponding to the $d$-stochastic product measure $\mu_A \otimes \lambda$, 
modeling independence of $\Xvec \sim A$ and $Y \sim \lambda$. 
Although our focus is on the linkage class, we formulate and prove our main results for the general setting $\kc^d_A$ with arbitrary $A \in \kc^{d-1}$.

\subsection{Bounds for the $L^p$-Deviation from the mean}
In general, the magnitude of $\normalt{r_C-\frac{1}{2}}_{A, p}$ substantially depends on the copula $C \in \CCal_A^d$. An exception, however, occurs under complete dependence in the sense 
of \cite[Lemma 5.4]{Griessenberger2022}, i.e., in the case in which there exists some $\mu_A$-$\lambda$-preserving transformation 
$h:\Ibb^{d-1} \rightarrow \Ibb$ such that $K(\xvec,F)=\ONE_{F}(h(\xvec))$ is (a version of) the Markov kernel 
of $C$. Again\textcolor{blue}{,} following \cite{Griessenberger2022}\textcolor{blue}{,} we will let $C_h \in \CCal_A^d $ denote the completely dependent 
copula induced by $h$.  

\begin{lem}[complete dependence] \label{Lem2025051301}
Suppose that $d \geq 2$ and let $C_h \in \CCal_A^d$ denote the completely dependent copula 
induced by the $\mu_A$-$\lambda$-preserving transformation $h:\Ibb^{d-1} \rightarrow \Ibb$. Then, $r_{C_h}(\xvec) = h(\xvec)$ for $\mu_A$-a.e. $\xvec \in \Ibb^{d-1}$, and for every $ p \in [1, \infty) $ we have
\begin{align*}
\norm{ r_{C_h} - \tfrac{1}{2} }_{A, p} = \tfrac{1}{2} (p+1)^{-\frac{1}{p}}.
\end{align*}
\end{lem}

\begin{proof}
In the described setup we have $K_{C_h}(\xvec, [0, y]) = \ONE_{[0, y]}(h(\xvec))$, which, using 
eq. (\ref{2025051303}) directly yields 
$$
r_{C_h}(\xvec) = \int_{ \Ibb } K_{C_h}(\xvec, (y,1] )\de \lambda(y) = h(\xvec).
$$ 
Moreover, applying change of coordinates and using the fact that the push-forward $\mu_A^h$ of $\mu_A$ via
$h$ coincides with $\lambda$ it follows that
\begin{align*}
\norm{r_{ C_h }-\tfrac{1}{2}}_{A, p}^p &= \intl_{\Ibb^{d-1}} \abs{ h(\xvec) - \tfrac{1}{2} }^p \de \mu_A(\xvec) = \intl_{ \Ibb } \abs{ y - \tfrac{1}{2} }^p \de\mu_A^h(y) \\
&= \intl_{ \Ibb } \abs{ y - \tfrac{1}{2}}^p \de\lambda(y).
\end{align*}
Elementary computations show $\int_{ \Ibb } |y-\frac{1}{2}|^p \de\lambda(y) = \frac{1}{2} (p+1)^{-\frac{1}{p}}$, 
and the proof is complete.
\end{proof}

It turns our that completely dependent copulas exhibit maximal average deviation 
from $\frac{1}{2}$: 

\begin{thm}[upper bound] \label{Theo2025070701}
For every $d \geq 2$ and every $p \in [1, \infty)$\textcolor{blue}{,} we have 
\begin{align*}
\maxl_{ C \in \CCal_A^d } \norm{ r_C - \tfrac{1}{2} }_{A, p} = \tfrac{1}{2} (p+1)^{-\frac{1}{p}}.
\end{align*}
\end{thm}

\begin{proof}
By definition of $r_C$ and since $\int_{ \Ibb } K_C(\xvec, \de y) = 1$, for 
each $\xvec \in \Ibb^{d-1}$ we can write 
$$
r_C(\xvec) - \tfrac{1}{2} = \int_{ \Ibb } \rb{y-\tfrac{1}{2}} K_C(\xvec, \de y)\textcolor{blue}{,}
$$ 
for every $\xvec \in \Ibb^{d-1}$. 
Applying Jensen's inequality yields 
$$
| r_C(\xvec) - \tfrac{1}{2} |^p \leq \int_{ \Ibb } | y-\tfrac{1}{2} |^p K_C(\xvec, \de y)\textcolor{blue}{.}
$$
Hence, using disintegration, we altogether conclude that
\begin{align*}
\norm{ r_C - \tfrac{1}{2} }_{A, p}^p &\leq \intl_{ \Ibb^{d-1} } \intl_{ \Ibb } \abs{ y-\tfrac{1}{2} }^p K_C(\xvec, \de y) \de\mu_A(\xvec) \\
&= \intl_{ \Ibb} \abs{ y-\tfrac{1}{2} }^p \de\lambda(y) = 
\tfrac{1}{p+1} \, 2^{-p}. 
\end{align*}
This shows that for every copula $C \in  \CCal_A^d $ the norm $\norm{ r_C - \tfrac{1}{2} }_{A, p}$ is 
bounded from above by $\tfrac{1}{2} (p+1)^{-\frac{1}{p}}$. Since  Lemma \ref{Lem2025051301} shows that 
this bound is attained, the proof of the theorem is complete.
\end{proof}

The previous result allows to quantify the maximum distance between two regression functions associated with copulas from the $\CCal_A^d$ family.

\begin{cor} \label{Cor2025070701}
For each $d \geq 2$ and every $p \in [1, \infty)$ the following identity holds:
\begin{align*}
\maxl_{C_1, C_2 \in \CCal_A^d} \norm{ r_{C_1} - r_{C_2} }_{A, p} = (p+1)^{-\frac{1}{p}}.
\end{align*}
\end{cor}

\begin{proof}
The triangle inequality implies 
$$\norm{ r_{C_1} - r_{C_2} }_{A, p} \leq \norm{ r_{C_1} - \tfrac{1}{2} }_{A, p} + \norm{ r_{C_2} - \tfrac{1}{2} }_{A, p},$$ so applying  Theorem \ref{Theo2025070701} we find that $\norm{ r_{C_1} - r_{C_2} }_{A, p}^p \leq (p+1)^{-1}$. Moreover, for $C \in \CCal_A^d$, obviously also $\overline C \in \CCal_A^d$, with $\norm{ r_C - r_{\overline C} }_{A, p} = 2\norm{ r_C - \frac{1}{2} }_{A, p}$. Hence, Lemma \ref{Lem2025051301} confirms the sharpness of the asserted bound.
\end{proof}

Our next example shows that completely dependent copulas are not the only copulas
attaining the upper bound in Theorem \ref{Theo2025070701}. 

\begin{ex} \label{Examp20250519}
Assume $N \geq 2$, let $\sigma^N$ be a permutation of $\curbr{1, \hdots, N}$, and consider the (checkerboard)  copula $C_N^\square$ with density $c_N^\square: \Ibb^{d} \rightarrow \Ibb$, given by 
\begin{align*}
c_N^\square(\xvec,y) := N \suml_{i=1}^N \ONE_{ I_{N, i} \times I_{ N, \sigma^N(i) } }(x_1,y). 
\end{align*}
Notice that $c_N^\square(\xvec,y)$ does not depend on $x_2,\ldots,x_{d-1}$ and that $C_N^\square$ is an element 
of the linkage family $\CCal_A^\Pi$.  
Considering that for $x_1 \in I_{N, i}$, the measure $K_{C_N^\square}(\xvec,\cdot)$ coincides with 
the uniform distribution on the interval $I_{N, \sigma^N(i)}$\textcolor{blue}{,}
it follows immediately that the regression function $r_{ C_N^\square }$ is given by
\begin{align*}
r_{ C_N^\square }(\xvec) = \tfrac{1}{N} \suml_{i=1}^N \ONE_{ I_{N, i} }(x_1) \rb{\sigma^N(i) - \tfrac{1}{2}}.
\end{align*}
Thus, for every $p \in [1, \infty)$\textcolor{blue}{,} we get
$$
\normalt{ r_{ C_N^\square } - \tfrac{1}{2} }_{\Pi, p}^p = N^{-p-1} \sum_{i=1}^N |i - \tfrac{N+1}{2}|^p.
$$
Specifically, for $p=1$\textcolor{blue}{,} we have  
$$
\normalt{ r_{ C_N^\square } - \tfrac{1}{2} }_{\Pi, 1} = \tfrac{1}{4} \ONE_{ \curbr{ N \in 2\N }} + \tfrac{N^2-1}{4N^2}\ONE_{ \curbr{ N \in 2\N_0+1} },
$$
so for even $N$ the upper bound from Theorem \ref{Theo2025070701} is attained. 
In general, using the Euler-Maclaurin summation formula shows that
\begin{align*}
    \normalt{ r_{ C_N^\square } - \tfrac{1}{2} }_{\Pi, p} = \tfrac{1}{2} (p+1)^{-\frac{1}{p}} + o(1) \hspace{1cm} (N \rightarrow \infty).
\end{align*}
\end{ex}

\begin{ex}[cube copula]\label{ex:cube}
    Consider the so-called cube copula $C^{\textrm{cube}} \in \kc^3_\Pi$ 
    introduced in \cite[Example 3.4]{MrozEtAl2021}, i.e., the three-dimensional copula 
     distributing its mass uniformly on the cubes $I_{2, 1}^3$, $I_{2,2}^2 \times I_{2,1}$, $I_{2,2} \times I_{2,1} \times I_{2, 2}$ and $I_{2, 1} \times I_{2, 2}^2$. 
     For this copula each conditional distribution $K_{C^{\textrm{cube}}}(\xvec,\cdot)$ 
     is either the uniform distribution on $I_{2,1}$ or on $I_{2,2}$. 
    Therefore, it is straightforward to verify that the regression function 
    $r_{C^{\textrm{cube}}}$ is given by
    \begin{align*}
        r_{ C^{\textrm{cube}} }(\xvec) = \tfrac{1}{4} \ONE_{ I_{2, 1}^2 \cup I_{2, 2}^2 }(\xvec) + \tfrac{3}{4} \ONE_{ (I_{2, 2} \times I_{2, 1}) \cup (I_{2, 1} \times I_{2, 2}) }(\xvec) \hspace{1cm} (\xvec \in \Ibb^2).
    \end{align*}
    According to \cite[Section 4.2]{MrozEtAl2021} the partial vine copula $\psi(C^{\textrm{cube}})$ of $C^{\textrm{cube}}$ is the independence/product copula $\Pi_3$,
    whose regression function is $r_{\Pi_3} \equiv \frac{1}{2}$. 
    As a direct consequence, we have 
    $$
    \normalt{ r_{ C^{\textrm{cube}} } - r_{\psi(C^{\textrm{cube}})} }_{\Pi, 1} =  \normalt{ r_{ C^{\textrm{cube}} } - \tfrac{1}{2} }_{\Pi, 1} =
    \tfrac{1}{4}.
    $$
    Considering that the according to Corollary \ref{Cor2025070701} the maximal $L^1$-distance of
    regression functions is at most $\frac{1}{2}$, the cube examples shows that 
    the partial vine copula (which, in the context of pair copula constructions is 
    commonly used as natural approximation of the original copula) can be a strikingly far off  also from the regression perspective. In fact, for the cube example 
    the error is 50\% of the maximum possible distance. 
    This simple observations underlines once more that working with partial vine 
    copulas must be done with care, despite the frequently praised `flexibility' 
    (see \cite{AasCzadoBakken2009,MrozEtAl2021,NAGLER} and the references therein). 
\end{ex}

We conclude this section with showing that the deviation from the mean can only decrease with reducing dimension. 
\begin{thm}[dimension reduction] \label{Theo20250513}
For each $d \geq 2$, $C \in \CCal^d$ and $p \in [1, \infty)$ the following inequality holds: 
\begin{align*}
\norm{ r_C - \tfrac{1}{2} }_{C_{1:(d-1)}, p} \geq \norm{ r_{C_{1:(d-2), d}} - \tfrac{1}{2} }_{C_{1:(d-2)}, p}
\end{align*}
\end{thm}

\begin{proof}
Fix $p \in [1, \infty)$ and $d \geq 2$. Then, using disintegration, we obtain
\begin{align*}
\norm{ r_C - \tfrac{1}{2} }_{C_{1:(d-1)}, p}^p = \intl_{ \Ibb ^{d-2}} I(\xvec_{1:d-2}) \de \mu_{ C_{1:(d-2)} }(\xvec_{1:d-2})
\end{align*}
with 
\begin{align*}
I(\xvec_{1:d-2}) := \intl_{ \Ibb } \abs{ r_C(\xvec_{1:d-2}, x_{d-1}) - \tfrac{1}{2} }^p K_{ C_{1:(d-1) } }(\xvec_{1:d-2}, \de x_{d-1}).
\end{align*}
Since $x \mapsto |x|^p$ is a convex function on $[-1, 1]$, Jensen's inequality yields 
\begin{align*}
I(\xvec_{1:d-2}) \geq \abs{ \intl_{ \Ibb } \rrb{ r_C(\xvec_{1:d-2}, x_{d-1}) - \tfrac{1}{2} } K_{ C_{1:(d-1)} }(\xvec_{1:d-2}, \de x_{d-1}) }^p.
\end{align*}
On the one hand, we obviously have $\int_{ \Ibb } K_{ C_{1:(d-1)} }(\xvec_{1:d-2}, \de x_{d-1}) = 1$. 
On the other hand, representing the regression function through eq. (\ref{2025051303}) and using Fubini's theorem,
we obtain
\begin{align*}
I(\xvec_{1:d-2}) \geq \abs{ \intl_{ \Ibb } \intl_{ \Ibb } K_C( \xvec_{1:d-2}, x_{d-1}, (y, 1] ) K_{ C_{1:(d-1)} }(\xvec_{1:d-2}, \de x_{d-1}) \de \lambda(y) - \tfrac{1}{2} }^p.
\end{align*}
Hence, from the disintegration identity (\ref{2025070202}), we infer that
\begin{align*}
\norm{ r_C - \tfrac{1}{2} }_{C_{1:(d-1)}, p}^p &\geq \intl_{ \Ibb }  \abs{ \intl_{ \Ibb }  K_{C_{1:(d-2), d}}( \xvec_{1:d-2}, (y, 1] ) \de \lambda(y) - \tfrac{1}{2} }^p \de\mu_{ C_{1:(d-2)} }(\xvec_{1:d-2}) \\
&= \norm{ r_{ C_{1:(d-2), d} } - \tfrac{1}{2} }_{C_{1:(d-2)}, p}^p,
\end{align*}
and the proof is complete.
\end{proof}

\subsection{Best-possible bounds for the distribution function of the $\overline m_{ \abs{ r_C - \frac{1}{2} } ,C}$}
Considering the function $f: \Ibb^{d-1} \rightarrow \Ibb$, given by $f(\xvec)=\vert  r_C(\xvec)-\frac{1}{2}\vert$, 
in this section we will study properties of the function 
\begin{align*}
a \mapsto \overline m_{ \abs{ r_C - \tfrac{1}{2} },C }(a) &= \mu_{C_{1:(d-1)}}( \cb{\abs{ r_C - \tfrac{1}{2} }}_a) \\
&= \mu_{C_{1:(d-1)}} \left(\{\xvec \in \Ibb^{d-1}: \abs{r_c(\xvec)-\tfrac{1}{2}} \geq a \}\right),
\end{align*} 
for $a\in [0,\frac{1}{2}]$. In other words, we study the (right-continous version of the) 
survival function of the random variable 
$\abs{ r_C - \tfrac{1}{2} }$ on the probability space $(\Ibb^{d-1},\mathcal{B}(\Ibb^{d-1}),\mu_{C_{1:(d-1)}})$.
In the sequel we will simply write $m_{ \abs{ r_C - \frac{1}{2} }}$ instead of 
$m_{ \abs{ r_C - \frac{1}{2} },C }$ since the dependence on $C$ is already indicated by $r_C$ and no confusion 
will arise. 
Considering 
\begin{align*}
\cb{ \abs{ r_C - \tfrac{1}{2} } }_a = r_C^{-1}\rb{ \cb{0, \tfrac{1}{2}-a} } \cup r_C^{-1}\rb{ \cb{\tfrac{1}{2}+a, 1} },
\end{align*}
and using $r_C^{-1}( [\frac{1}{2}+a, 1] ) = [r_C]_{a+\frac{1}{2}}$ as well as 
$r_C^{-1}( [0, \frac{1}{2}-a] ) = [-r_C]_{a-\frac{1}{2}}$ directly
yields the following alternative expression for $\overline m_{ \abs{ r_C - \frac{1}{2} } }(a)$, which 
we will use in some of the proofs:
\begin{align} \label{2025070901}
\overline m_{ \abs{ r_C - \tfrac{1}{2} } }(a) = \overline m_{ r_C }\rb{ a+\tfrac{1}{2} } + \overline m_{ -r_C }\rb{ a-\tfrac{1}{2} } .
\end{align}
As for the $L_p$-norm we first have a look at the completely dependent case.
\begin{ex}\label{ex:cd.not}
Let $d \geq 2 $ and suppose that $C_h \in \CCal_A^d$ is completely dependent with $\mu_A$-$\lambda$-preserving 
transformation $h : \Ibb^{d-1} \rightarrow \Ibb$. Then\textcolor{blue}{,} according to Lemma \ref{Lem2025051301}\textcolor{blue}{,} we have 
$r_C(\xvec) = h(\xvec)$, so for every $a \in [0, \frac{1}{2}]$
$$
\overline m_{ | r_{ C_h } - \frac{1}{2} | }(a) = 1-2a
$$
follows immediately.
\end{ex}

As we are going to show, in the bivariate case it turns out that a best-possible 
upper bound for $\overline m_{ | r_{ C_h } - \frac{1}{2} | }(a)$
 can be obtained via a slight modification of the Hardy-Littlewood-Pólya theorem (see, e.g,  
 \cite[Ch. 1, Theorem D.2]{Marshall2011}). 
After having shown the result in the bivariate setting\textcolor{blue}{,} 
we will then tackle its extension to arbitrary dimension $d \geq 3$\textcolor{blue}{,} using a measure-isomorphism 
argument in the following sense: 
It is well known (see \cite[Theorem 2.1]{Walters1982}) that for arbitrary but fixed 
$A \in \kc^{d-1}$, the probability spaces $(\Ibb,\mathcal{B}(\Ibb),\lambda_A)$ and 
$(\Ibb^{d-1},\mathcal{B}(\Ibb^{d-1}),\mu_{A})$ are isomorphic (since the latter has no point masses). To be precise,
there exist some Borel sets $\Lambda_1 \in \mathcal{B}(\Ibb)$ and 
$\Lambda_{d-1} \in \mathcal{B}(\Ibb^{d-1})$ with 
$\lambda(\Lambda_1)=\mu_{A}(\Lambda_{d-1})=1$ and a measurable (with respect to the trace $\sigma$-fields) bijection
\begin{align} \label{def:iota}
    \iota_A: \Lambda_{d-1} \rightarrow \Lambda_{1},
\end{align}
such that the push-forward $\mu_{A}^{\iota_A}$ of $\mu_{A}$ 
via $\iota_A$ coincides with $\lambda$.
%(the latter w.r.t.t. corresponding trace $\sigma$-fields on $\Lambda_1$ and $\Lambda_{d-1}$). 
Extending $\iota_A$ to an $\Ibb$-valued measurable transformation on $\Ibb^{d-1}$ 
(by, e.g., setting $\iota_A(\xvec)=0$ for all $\xvec \in \Ibb^{d-1} \setminus \Lambda_{d-1}$) 
and defining the mapping $\Psi_A: \Ibb^d \rightarrow \Ibb^2$ by 
\begin{align} \label{def:psiA}
    \Psi_A(\mathbf{x},y)= (\iota_A(\mathbf{x}),y),
\end{align}
the following simple lemma holds.

\begin{lem} \label{Lem2025091501}
For every fixed $A \in \kc^{d-1}$ and every $C \in \kc_A^d$\textcolor{blue}{,} the measure $\mu_C^{\Psi_A}$ 
is doubly stochastic. In particular, $\Psi_A$ can be interpreted as mapping
from $\kc_A^d$ to $\kc^2$.
\end{lem}
\begin{proof}
Since the push-forward $\mu_C^{\Psi_A}$ obviously is a probability measure on 
$\mathcal{B}(\Ibb^2)$, it suffices
to show that $\mu_C^{\Psi_A}$ is doubly stochastic. Letting $E \in \mathcal{B}(\Ibb)$ be fixed, 
we have 
\begin{align*}
\mu_C^{\Psi_A}(E \times \Ibb) &= \mu_C\left(\iota_A^{-1}(E) \times \Ibb\right) 
= \mu_A(\iota_A^{-1}(E)) = \mu_A^{\iota_A}(E)= \lambda(E).
\end{align*} 
Since the property $\mu_C^{\Psi_A}(\Ibb \times E) = \lambda(E)$ is obvious, the proof is complete.
\end{proof}

Building upon the previous lemma we now prove the following main result of this section.
\begin{thm} \label{Theo20250820}
For each $d \geq 2, A \in \kc^{d-1}$ and $C \in \CCal_A^d$ the following inequality holds for every 
$a \in \cb{0, \frac{1}{2}}$: 
\begin{align*}
\overline m_{ | r_C - \frac{1}{2} | }(a) \leq \min\curbr{1, 2-4a} 
\end{align*}
\end{thm}

\begin{proof}
We proceed in two steps, first derive the result for $d=2$ and then extend to general dimension $d\geq 3$.  
(i) First, consider $C \in \CCal^2$ and set $q(y) := y$ for every $y \in \Ibb$. 
According to Jensen's inequality for conditional expectation, for any convex $\varphi : \R \rightarrow \R$
and $(X,Y) \sim C$, it holds that $\varphi(r_C(X)) = \varphi(\CExp{Y}{X}) \leq \CExp{\varphi(Y)}{X}$, implying $\Exp{\varphi(r_C(X))} \leq \Exp{\varphi(q(Y))}$, i.e., $r_C(X)$ is dominated by $q(Y)=Y$ in convex order.
Consequently, referring to the Hardy-Littlewood-Polya theorem (see \cite[Ch. 1, Theorem D.2]{Marshall2011}), we
conclude that
\begin{align*}
\intl_{ [0, t] } r_{C, \downarrow}(u) \de\lambda(u) \leq \intl_{ [0, t] } q_\downarrow(u) \de\lambda(u) 
\end{align*}
for every $t \in \Ibb$. Moreover, considering $m_q(v) = 1-v$ and $q_\downarrow(u) = 1-u$ yields
\begin{align} \label{2025082101}
\intl_{ [0, t] } r_{C, \downarrow}(u) \de \lambda(u) \leq t - \tfrac{t^2}{2} 
\end{align}
for every $t \in \Ibb$.
At the same time, recalling that $r_{C, \downarrow} : \Ibb \rightarrow \Ibb$ is decreasing, obviously 
$r_{C, \downarrow}(u) \geq \frac{1}{2}+a$ for every 
$u \in [0, \overline m_{ r_C }(a+\frac{1}{2} ))$. Considering $t:= \overline m_{ r_C }(a+\frac{1}{2} )$
in eq. (\ref{2025082101}) we thus arrive at
\begin{align*}
\rb{\tfrac{1}{2}+a} \overline m_{ r_C }\rb{ a+\tfrac{1}{2} } \leq \overline m_{ r_C }\rb{ a+\tfrac{1}{2} } - \tfrac{1}{2} \rrb{ \overline m_{ r_C }\rb{ a+\tfrac{1}{2} } }^2,
\end{align*}
which implies that $\overline m_{ r_C }( a+\frac{1}{2} ) \leq 1-2a$ holds for every 
$a \in [0, \frac{1}{2}]$. Regarding $\overline m_{ -r_C }( a-\frac{1}{2} )$, we can write
\begin{align*}
\overline m_{ -r_C }\rb{ a-\tfrac{1}{2} } = \overline m_{ 1-r_C }\rb{ a+\tfrac{1}{2} }.
\end{align*}
In addition, $\Exp{\varphi(1-r_C(X))} \leq \Exp{\varphi(1-Y)}$ for every convex $\varphi$. 
Therefore, proceeding analogously to the first part, it follows that 
$\overline m_{ -r_C }( a-\frac{1}{2} ) \leq 1-2a$. Finally, considering eq. (\ref{2025070901}) 
it altogether follows that $$\overline m_{ | r_C - \frac{1}{2} | }(a) \leq \min\curbr{1, 2-4a},$$ which 
completes the proof for $d=2$. \\
(ii) Suppose now that $d \geq 3$ and that $C \in \kc_A^d$ is arbitrary but fixed. 
Interpreting $\Psi_A$ as mapping from $\kc_A^d$ to $\kc^2$, Lemma \ref{Lem2025091501} implies that 
for $\mu_A$-almost every $\mathbf{x} \in \Lambda_{d-1}$
$$
K_C(\mathbf{x},\cdot) = K_{\Psi_A(C)}(\iota_A(\mathbf{x}),\cdot)
$$
holds. This, however, yields that for every $E \in \mathcal{B}(\Ibb)$ we have 
\begin{equation}
\mu_A\left(\{\mathbf{x} \in \Ibb^{d-1}: r_C(\mathbf{x}) \in E \}\right)
= \lambda\left(\{x \in \Ibb: r_{\Psi_A(C)}(x) \in E \}\right).
\end{equation}
As a direct consequence, the two functions $\overline m_{ |r_C - \frac{1}{2}| }$ and 
$\overline m_{ |r_{\Psi_A(C)} - \frac{1}{2}| }$ coincide, so the desired inequality follows from case (i). 
\end{proof}

Completely dependent copulas turned out to maximize the $L^p$-norm, considering Example \ref{ex:cd.not},  
however, they apparently do not attain the upper bound according to Theorem \ref{Theo20250820}.
Nevertheless, as the following example shows, the upper bound can be attained.   
\begin{ex} \label{Ex20250721}
Let $A \in \kc^{d-1}$ be arbitrary but fixed. 
For $b \in (0, \frac{1}{2})$, let $O_2 \in \CCal^2$ denote the ordinal sum (see 
\cite[$\S$3.8]{DS2016}) of the
independence copula $\Pi$ with respect to the segments $(0, b)$, $(b, 1-b)$ and $(1-b, 1)$ (see Figure \ref{ordinal_sum_plot}). 
Setting 
$$
K(\xvec, [0,y]) := K_{O_2}(x_1,[0,y])
$$
for all $(\xvec, y) \in \Ibb^d$ obviously defines a Markov kernel of a copula $O_A \in \CCal_A^d$, given by
$$
O_A(\xvec,y)= \int_{[\mathbf{0}, \xvec]} K(\mathbf{t}, [0,y]) \de \mu_A(\mathbf{t})
$$
for every $(\xvec,y) \in \Ibb^d$. By construction, the regression function $r_C$ of 
$C$ is given by 
$$
r_C(\xvec) = \tfrac{b}{2} \ONE_{(0, b)}(x_1) + \tfrac{1}{2} \ONE_{(b, 1-b)}(x_1) + 
(1-\tfrac{b}{2}) \ONE_{(1-b, 1)}(x_1).
$$
So, in particular we have 
$$
|r_C(\xvec) - \tfrac{1}{2}| = \tfrac{1-b}{2}\ONE_{\Ibb \setminus (b,1-b)}(x_1)
$$
for all $\xvec\in \Ibb^{d-1}$, which directly yields
\begin{align*}
    \overline m_{ |r_C(x) - \frac{1}{2}| }(a) = \ONE_{\curbr{a=0}} + 2b \ONE_{ (0, \frac{1-b}{2}] }(a) 
\end{align*}
for $ a \in \cb{0, \frac{1}{2}}$.
Considering $a \in (\frac{1}{4}, \frac{1}{2}]$ and setting $b = 1-2a$, the upper bound from Theorem 
\ref{Theo20250820} is attained. 
To show exactness for $a \in [0, \frac{1}{4}]$, instead of $O_2$ consider a checkerboard copula 
$C_2^\square$ according to Example \ref{Examp20250519}. 
This copula fulfills $|r_{C_2^\square}(\xvec) - \frac{1}{2}| = \frac{1}{4} \geq a$ for all $\xvec \in \Ibb^{d-1}$. %Indeed, one can even show that
%\begin{align*}
%    \supl_{N \in \N} \overline m_{ | r_{C_N^\square}-\frac{1}{2}| }(a) = \ONE_{[0, \frac{1}{4}]}(a) + (2-4a) %\ONE_{(\frac{1}{4}, \frac{1}{2}]}(a) \hspace{1cm} \rb{ a \in \cb{0, \frac{1}{2}}}.
%\end{align*}
\end{ex}

\begin{figure}
\centering
\resizebox*{0.5\textwidth}{!}{\includegraphics{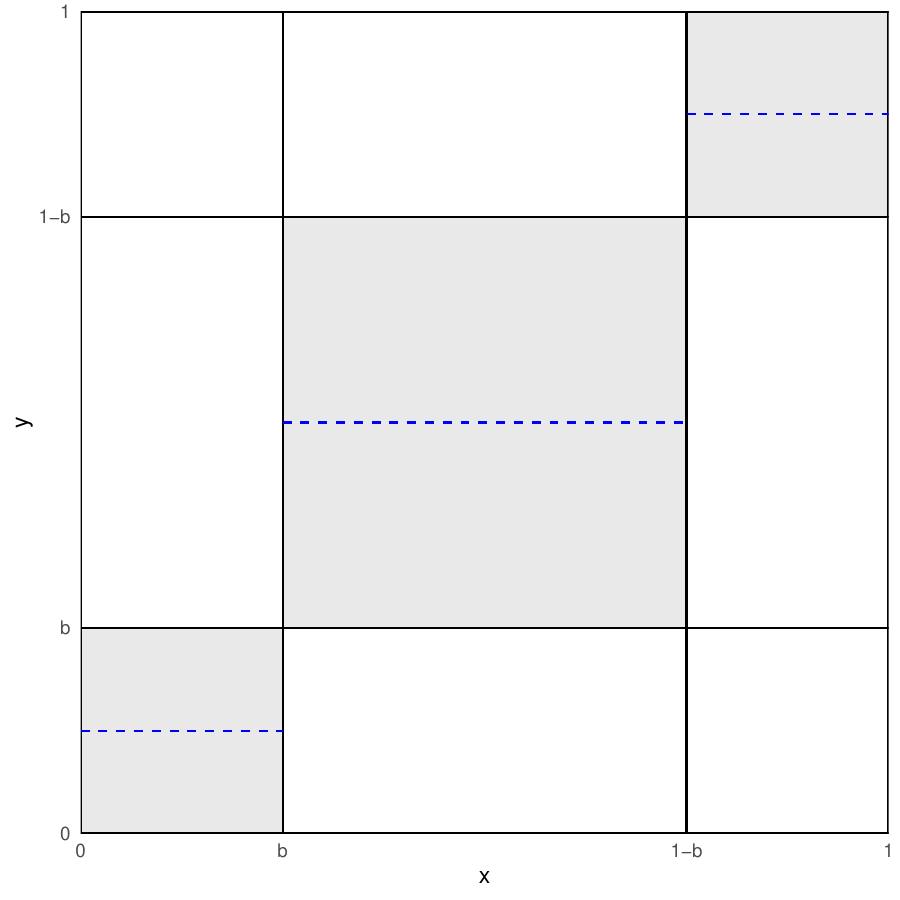}}
\caption{Support of the ordinal sum $O_2$ considered in Example \ref{Ex20250721} (gray) and
regression function $r_{O_2}$ (dashed blue line).}
\label{ordinal_sum_plot}
\end{figure}

Combining Theorem \ref{Theo20250820} and Example \ref{Ex20250721} yields the following corollary:
\begin{cor} \label{Cor20250910}
For every $d \geq 2$, every $A \in \kc^{d-1}$ and every $a \in \cb{0, \frac{1}{2}}$ we have
\begin{align*}
\maxl_{ C \in \CCal_A^d } \overline m_{ | r_C - \frac{1}{2} | }(a) = \min\curbr{1, 2-4a}.
\end{align*}
\end{cor}

\subsection{Relation to the metric $D_p$}

We extend the $D_p$-metric according to eq. (\ref{2025062701}) to  $\CCal_A^d$ by setting 

\begin{align} \label{2025093001}
    D_{A, p}(C_1, C_2) := \rrb{ \intl_{ [0, 1] } \Phi_{C_1, C_2; A, p}(y) \de \lambda(y) }^\frac{1}{p} \hspace{1cm} (p \in [1, \infty)),
\end{align}
for all $C_1, C_2 \in \CCal_A^d$ where, analogously to (\ref{2025062702}), we define
\begin{align} \label{2025093002}
\Phi_{C_1, C_2;A, p}(y) := \intl_{ \Ibb^{d-1} } \abs{ K_{C_1}(\xvec, [0, y]) - K_{C_2}(\xvec, [0, y]) }^p \de \mu_A(\xvec)
\end{align}
for every $y \in \Ibb$. 
According to \cite{Griessenberger2022} $D_{\Pi, p}(C_1, C_2) = D_p(C_1, C_2)$ 
is a metric on $\CCal_\Pi^d$. The following more general result holds: 

\begin{lem} \label{Lem20250630}
    For each $d \geq 2$, every $p \in [1, \infty)$, and $A \in \CCal^{d-1}$, the mapping $D_{A, p} : \CCal_A^d \times \CCal_A^d \rightarrow [0, 1]$ establishes a metric on $\CCal_A^d$. This metric fulfills
    \begin{align*}
    \norm{ r_{C_1}-r_{C_2} }_{A, p} \leq D_{A, p}(C_1, C_2) \hspace{1cm} (C_1, C_2 \in \CCal_A^d).
    \end{align*}
\end{lem}

\begin{proof}
It is an easy exercise to verify the metric properties. For the proof of the asserted inequality,  from (\ref{2025051303}) we directly get
\begin{align*}
\norm{ r_{C_1} - r_{C_2}}_{A, p}^p = \intl_{ \Ibb^{d-1} } \abs{ \intl_{ \Ibb } \rb{ K_{C_1}(\xvec, [0, y]) - K_{C_2}(\xvec, [0, y]) } \de \lambda(y) }^p \de\mu_A(\xvec).
\end{align*}
Thus, an application of Jensen's inequality yields
\begin{align*}
\norm{ r_{C_1} - r_{C_2}}_{A, p}^p \leq \intl_{ \Ibb^{d-1} } \intl_{ \Ibb } \abs{ K_{C_1}(\xvec, [0, y]) - K_{C_2}(\xvec, [0, y]) }^p \de\lambda(y) \de\mu_A(\xvec),
\end{align*}
and the proof is complete.
\end{proof}

An immediate consequence of Lemma \ref{Lem20250630} is the following result stating
that for $L_p$-convergence of regression functions of elements in $\kc^d_A$ 
weak convergence of $\mu_A$-almost all conditional distributions is not necessary - 
convergence w.r.t. $D_{A_,p}$ suffices. 

\begin{cor}
$D_{A, p}$-convergence of copulas in $\CCal_A^d$ implies $L^p$-convergence of the associated regression functions.
\end{cor}

The function $\Phi_{C_1, C_2} \equiv \Phi_{C_1, C_2; \Pi, 1}$ has already
been studied in \cite{Trutschnig2017} and \cite{Griessenberger2022}
in the bivariate and the multivariate linkage setting, respectively. 
The subsequent lemma summarizes the analogous properties for the function $\Phi_{C_1, C_2;A, p}$. 

\begin{lem} \label{Lem2025063002}
Suppose that $d \geq 2$, that $A \in \CCal^{d-1}$, and that 
$p \in [1, \infty)$. Then for arbitrary $C_1, C_2 \in \CCal_A^d$ the 
function $\Phi_{C_1, C_2; A, p}$ defined according to eq. (\ref{2025093002})
has the following properties: 
\begin{enumerate}
    \item $\Phi_{C_1, C_2; A, p}(0) = \Phi_{C_1, C_2; A, p}(1) = 0$
    \item $\Phi_{C_1, C_2; A, p}$ is continuous on $\Ibb$, for $p=1$ even $2$-Lipschitz.
    \item $\Phi_{C_1, C_2; A, p}(y) \leq 2\min\curbr{y, 1-y}$ for every $y \in \Ibb$ and this bound is sharp.
\end{enumerate}
Moreover the metric $D_{A,p}$ fulfills
\begin{align} \label{Lem2025063002Eq1}
\maxl_{C_1, C_2 \in \CCal_A^d} D_{A, p}(C_1, C_2) = 2^{ -\frac{1}{p} },
\end{align} 
i.e., the diameter of the metric space $(\kc^d_A,D_{A,p})$ is $2^{ -\frac{1}{p} }$.
\end{lem}

\begin{proof}
The boundary behavior of $\Phi_{C_1, C_2; A, p}$ is obvious. 
To prove the remaining assertions we proceed as follows. 
Recall that, according to Lemma \ref{Lem2025091501}, for every 
$C \in \CCal_A^d$ the measure $\widetilde{C}:= \mu_C^{\Psi_A}$ is doubly stochastic. 
Again using the isomorphism $\iota_A$ and the mapping $\Psi_A$ according to eq. (\ref{def:psiA}), for each $C \in \CCal_A^d$ and 
$\mu_A$-almost every $\xvec \in \Ibb^{d-1}$ we have $K_C(\xvec, [0, y]) = K_{\widetilde C}(\iota_A(\xvec), [0, y])$. For $C_1, C_2 \in \CCal_A^d$ using change of coordinates 
we therefore get 
    \begin{align*}
        \Phi_{C_1, C_2;A, p}(y) &= \intl_{ \Ibb^{d-1} } \abs{ K_{C_1}(\xvec, [0, y]) - K_{C_2}(\xvec, [0, y]) }^p \de \mu_A(\xvec) \\
        &= \intl_{ \Ibb^{d-1} } \abs{ K_{ \widetilde C_1 }(\iota_A(\xvec), [0, y]) - K_{ \widetilde C_2 }(\iota_A(\xvec), [0, y]) }^p \de \mu_A(\xvec) \\
         &= \intl_{ \Ibb^{d-1} } \abs{ K_{\widetilde C_1}(u, [0, y]) - 
                                 K_{\widetilde C_2}(u, [0, y]) }^p \de \lambda(u) \\
        &= \Phi_{\widetilde C_1, \widetilde C_2;\lambda, p}(y).
    \end{align*}
Lipschitz-continuity of $y \mapsto \Phi_{C_1, C_2 ;A, 1}(y)$ is now a direct consequence 
of the bivariate seeting analyzed in \cite[Lemma 5]{Trutschnig2017}. 
For general $p>1$, as a direct consequence of Dominated Convergence, 
$\Phi_{C_1, C_2 ;A, p}$ is right-continuous. Moreover, using the fact that we have 
$\mu_{C_1}(\Ibb^{d-1} \times \{s\})=0=\mu_{C_2}(\Ibb^{d-1} \times \{s\})$ 
for every $s \in \Ibb$, it follows that the function $y \mapsto \Phi_{C_1, C_2 ;A, 1}$ 
is also left-continuous, which completes the proof of the second assertion. \\
Considering
$\Phi_{\widetilde C_1, \widetilde C_2;\lambda, p} \leq \Phi_{\widetilde C_1, \widetilde C_2;\lambda, 1}$ and again using \cite[Lemma 5]{Trutschnig2017}  
we have 
$\Phi_{C_1, C_2; A, p}(y) \leq 2y \ONE_{[0, \frac{1}{2}]}(y) + 2(1-y) \ONE_{(\frac{1}{2}, 1]}(y)$,
which yields the third assertion. 
Finally, by the very definition of $D_{A,p}$
$$
\rrb{ D_{A, p}(C_1, C_2) }^p \leq 2\int_{ [0, \frac{1}{2} ]} y \de \lambda(y) + 2\int_{ [\frac{1}{2}, 1] } (1-y) \de \lambda(y) = \tfrac{1}{2}
$$ 
follows and it only remains to show the existence of copulas $C_1,C_2 \in \kc^d_A$ fulfilling
$(D_{A, p}(C_1, C_2))^p=\frac{1}{2}$, which can be done as follows: 
Consider a completely dependent copula $C_h \in \CCal_A^d$ and its flipped counterpart 
$\overline C_h=C_{1-h} \in \CCal_A^d$. Then, using change of coordinates, we get 
\begin{align*}
\Phi_{C_h, C_{1-h}; A, p}(y) &= \intl_{ \Ibb^{d-1} } \abs{ \ONE_{ [0, y] }(h(\xvec)) - \ONE_{ [1-y, 1] }(h(\xvec)) }^p \de\mu_A(\xvec) \\
&= \intl_{ \Ibb } \abs{ \ONE_{ [0, y] }(u) - \ONE_{ [1-y, 1] }(u) }^p \de\lambda(u) \\
&= \intl_{ \Ibb } \abs{ \ONE_{ [0, y] }(u) - \ONE_{ [1-y, 1] }(u) }^p \de\lambda(u) = 
2\min\curbr{y, 1-y},
\end{align*}
which completes the proof. 
\end{proof}

Combining Lemma \ref{Lem20250630} and Lemma \ref{Lem2025063002} it follows that 
for all $d \geq 2$, $p \in [1, \infty)$ and $C_1, C_2 \in \CCal_A^d$
\begin{align*}
\norm{ r_{C_1} - r_{C_2} }_{A, p} \leq 2^{ -\frac{1}{p} }
\end{align*}
holds. 
This bound coincides with the bound from Corollary \ref{Cor2025070701} for $p=1$, while being 
too rough for $p>1$. 
%---------------------------------------------------------------------------------------------
\section{Quantile Regression}
After having derived various results on the regression function of copulas we now
turn towards quantile regression and derive various analogous results. Doing so, we start with 
some first observations concerning the extreme cases of independence and complete dependence
(in the linkage class).
\begin{ex}[Independence and complete dependence] \label{ExQuant20251003}
Let $d \geq 2$ as well as $\tau\in(0, 1]$ be arbitrary but fixed. 
Then obviously for the product copula
$\Pi$ we have $Q_\Pi^\tau(\xvec) = \tau$ for all $\xvec \in \Ibb^{d-1}$. \\
Considering the other extreme, let $h : \Ibb^{d-1} \rightarrow \Ibb$ denote a 
$\lambda_{d-1}$-$\lambda$-preserving transformation and $C_h$ the corresponding
induced completely dependent copula. In this case, $Q_{ C_h }^\tau(\xvec) = h(\xvec)$, for every $\xvec \in \Ibb^{d-1}$, implying that
$$
\intl_{\Ibb^{d-1}} Q_{C_h}^\tau(\mathbf{x})\de\mu_A(\mathbf{x}) = \intl_{\Ibb^{d-1}} h(\mathbf{x}) \de 
\lambda_{d-1}(\mathbf{x}) =  
\intl_{\Ibb} z \de \lambda_{d-1}^h(z)  = \intl_{\Ibb} z \de \lambda(z) =\tfrac{1}{2}.
$$
In other words: the quantile function $Q^\tau_{C_h}$ integrates to $\frac{1}{2}$ for every $\tau \in (0,1]$. 
\end{ex}
\noindent In the following we will derive (best-possible) inequalities for the average value 
$$\intl_{\Ibb^{d-1}} Q_C^\tau(\mathbf{x}) \de\mu_A(\mathbf{x})$$ 
of the quantile function $Q_C^\tau$ for $C \in \kc^d_A$. 
Doing so, we will use the following elementary identity (a generalization of which is usually referred to as 
the `layer cake representation', see  \cite[Theorem 1.13]{LiebLoss2001}): 
\begin{equation}\label{eq:fubini}
 \intl_{\Ibb^{d-1}} Q_C^\tau(\mathbf{x}) \de\mu_A(\mathbf{x}) = 
 \intl_{\Ibb} \overline m_{ Q^\tau_C }(q) \de\lambda(q) =
  \intl_{\Ibb} m_{ Q^\tau_C }(q) \de\lambda(q).
\end{equation}
\begin{thm}
For each $d\geq 2$, $A \in \CCal^{d-1}$, $C \in \CCal_A^d$ and $\tau \in (0, 1]$ 
the following inequality holds: \begin{align*}
\tfrac{\tau}{2} \leq \intl_{ \Ibb^{d-1} } Q_C^\tau(\xvec) \de\mu_A(\xvec) \leq \tfrac{\tau+1}{2}.
\end{align*}
Thereby, the lower and the upper bound are best possible. 
\end{thm}

\begin{proof}
First, using disintegration we obviously have  
\begin{eqnarray*}
q &=& \intl_{\Ibb^{d-1}} K_C(\mathbf{x},[0,q])\de\mu_A(\mathbf{x}) =  \intl_{\Ibb^{d-1}} K_C(\mathbf{x},[0,q))\de\mu_A(\mathbf{x}) \\ 
&=& \intl_{[Q^\tau_C]_q} K_C(\mathbf{x},[0,q))\de\mu_A(\mathbf{x}) 
+ \intl_{\Ibb^{d-1} \setminus [Q^\tau_C]_q} K_C(\mathbf{x},[0,q))\de\mu_A(\mathbf{x}).
\end{eqnarray*} 
For every $\mathbf{x} \in [Q^\tau_C]_q$ by construction of the $\tau$-quantile function 
we have $K_C(\mathbf{x},[0,y]) < \tau$ for every $y < q$, implying that $K_C(\mathbf{x},[0,q)) \leq \tau$. 
The first integral is therefore bounded from above by $\tau \cdot \overline m_{ Q_C^\tau }(q)$. 
Using the obvious upper bound $1-\overline m_{ Q_C^\tau }(q)$ for the second integral altogether 
yields
\begin{eqnarray*}
q \leq \tau \cdot \overline m_{ Q_C^\tau }(q) + 1-\overline m_{ Q_C^\tau }(q),
\end{eqnarray*}
from which it directly follows that
$$
\overline m_{ Q_C^\tau }(q) \leq \min \left\{1,\tfrac{1-q}{1-\tau}\right\}.
$$
Applying eq. (\ref{eq:fubini}) and calculating the integral directly 
yields the upper bound of $\frac{\tau+1}{2}$. \\
To prove the lower point we proceed as follows: 
Using the fact that $Q^\tau_C(\mathbf{x}) < q$ implies $K_C(\mathbf{x},[0,q]) \geq \tau$ it 
follows that 
\begin{eqnarray*}
 q &=& \intl_{\Ibb^{d-1}} K_C(\mathbf{x},[0,q])\de\mu_A(\mathbf{x}) \geq 
 \intl_{\Ibb^{d-1} \setminus [Q^\tau_C]_q} K_C(\mathbf{x},[0,q))\de\mu_A(\mathbf{x}) \\
 &\geq& \tau \mu_A\left(\Ibb^{d-1} \setminus [Q^\tau_C]_q\right) = \tau(1-\overline m_{ Q_C^\tau }(q)).
\end{eqnarray*}
This directly yields 
$$
\overline m_{ Q_C^\tau }(q) \geq \max \left\{0,1-\tfrac{q}{\tau}\right\}.
$$
for every $q \in \Ibb$. Again using eq. (\ref{eq:fubini}) and calculating the integral directly 
yields the lower bound $\frac{\tau}{2}$. \\
Finally, it remains to show that the established bounds
are best-possible. 
For fixed $\tau \in (0,1)$, defining $K: \Ibb^{d-1} \times \mathcal{B}(\Ibb) \rightarrow \Ibb$ 
by 
$$
K(\mathbf{x},E)= \tau \ONE_E(\tau x_1) + \, (1-\tau) \ONE_E(\tau+(1-\tau)x_1),
$$
obviously $K$ is the $(d-1)$-Markov kernel of a unique copula $C \in \kc^d_A$. 
For this very copula $C$ the $\tau$-quantile function $Q^\tau_C$, however, is given by 
$Q^\tau_X(\mathbf{x})=\tau x_1$, which yields 
$$
\intl_{\Ibb^{d-1}} Q^\tau_C(\mathbf{x}) \de\lambda_{d-1}(\mathbf{x}) = \intl_{\Ibb} \tau x_1 \, \de\lambda(x_1)
= \tfrac{\tau}{2},
$$
so the lower bound is attainable.  
For showing that the upper bound is best-possible, consider $n\in\mathbb{N}$ sufficiently large, so that $\tau - \frac{1}{n} \in (0,1)$ holds and set  
$$
K(\mathbf{x},E):= (\tau-\tfrac{1}{n}) \ONE_E((\tau-\tfrac{1}{n}) x_1) + \, (1-\tau+\tfrac{1}{n}) \ONE_E(\tau-\tfrac{1}{n} + (1-\tau + \tfrac{1}{n})x_1),
$$
for every $\mathbf{x} \in \Ibb^{d-1}$ and $E \in \mathcal{B}(\Ibb)$. 
Then, $K$ is the $(d-1)$-Markov kernel of a unique copula $C \in \kc^d_A$, whose  
the $\tau$-quantile function $Q^\tau_C$ obviously is given by 
$Q^\tau_C(\mathbf{x})=\tau-\tfrac{1}{n} + (1-\tau + \tfrac{1}{n})x_1=Q^1_C(\mathbf{x})$. 
A straightforward calculation yields
\begin{eqnarray*}
\intl_{\Ibb^{d-1}} Q^\tau_C(\mathbf{x}) \de\lambda_{d-1}(\mathbf{x}) = \tfrac{1}{2} + \tfrac{1}{2}(\tau-\tfrac{1}{n}),     
\end{eqnarray*}
hence, considering $n \rightarrow \infty$ completes the proof.
\end{proof}

Although the integral of the $\tau$-quantile function $Q^\tau_C$ 
does not need to coincide with the value $\tau$, as the following lemma shows, integrating 
over $\tau$ again yields the same constant for all copulas.

\begin{lem}
For each $d \geq 2$ and $C \in \CCal_A^d$, we have
\begin{align*}
\intl_{ \Ibb } \intl_{ \Ibb^{d-1} } Q_C^\tau(\xvec) \de\mu_A(\xvec) \de\lambda(\tau) = \tfrac{1}{2}.
\end{align*}
\end{lem}

\begin{proof}
Considering eq. (\ref{eq:fubini}) and using the fact that $Q^\tau_C(\mathbf{x})>q$ if, and only if
$K_C(\mathbf{x},[0,q])<\tau$, directly yields
\begin{align*}
\intl_{ \Ibb^{d-1} } Q_C^\tau(\xvec) \de\mu_A(\xvec) &= 
\intl_{ \Ibb } \mu_A\left( \curbr{\xvec \in \Ibb^{d-1} : K_C(\xvec, [0, q]) < \tau } \right) \de \lambda(q) \\
&= \intl_{ \Ibb } \intl_{\Ibb^{d-1}} \ONE_{[0,\tau)} (K_C(\xvec, [0, q])) \de \mu_A(\xvec) \de \lambda(q) \\
&=  \intl_{ \Ibb } \intl_{\Ibb^{d-1}} \ONE_{[K_C(\xvec, [0, q]),1]} (\tau) \de \mu_A(\xvec) \de \lambda(q). 
\end{align*}
Having this, applying Fubini's theorem and disintegration, we altogether get
\begin{align*}
\intl_{ \Ibb } \intl_{ \Ibb^{d-1} } Q_C^\tau(\xvec) \de\mu_A(\xvec) \de \lambda(\tau) &= 
\intl_{ \Ibb }  \intl_{ \Ibb } \intl_{ \Ibb^{d-1} } \ONE_{[K_C(\xvec, [0, q]),1]} (\tau) \de \mu_A(\xvec) \de \lambda(q) \de \lambda(\tau) \\
&= \intl_{ \Ibb } \intl_{ \Ibb^{d-1} }  \intl_{ \Ibb }\ONE_{[K_C(\xvec, [0, q]),1]} (\tau) \de \lambda(\tau) \de \mu_A(\xvec) \de \lambda(q) \\
&= \intl_{ \Ibb } \intl_{ \Ibb^{d-1} }  K_C(\mathbf{x},(q,1]) \de \mu_A(\xvec) \de \lambda(q) \\
&= \intl_{ \Ibb } (1-q) \de \lambda(q) = \tfrac{1}{2}
\end{align*}
and the proof is complete. 
\end{proof}
Proceeding analogous to the proof of the previous result, we conclude this section with a sharp upper bound for the average $L^1$-distance of quantile functions for copulas in the the 
family $\kc^d_A$ and a direct consequence to the cube copula.

\begin{thm}\label{thm:average.quantile}
For each $d \geq 2$, $A \in \kc^{d-1}$, and arbitrary $C_1,C_2\in \CCal_A^d$ the following 
inequality holds: 
\begin{align}
    D_{A,1}(C_1,C_2) = \intl_\Ibb \norm{ Q_{C_1}^\tau - Q_{C_2}^\tau}_{A,1} \de \lambda(\tau)  
    \leq \tfrac{1}{2}.
\end{align}
This inequality is best-possible.
\end{thm}

%\begin{lem}
%It holds that
%\begin{align*}
%    \maxl_{C_1, C_2 \in \CCal^2} \intl_0^1 \norm{ Q_{C_1}^\tau - Q_{C_2}^\tau }_1 \de \tau = \frac{1}{2}.
%\end{align*}
%\end{lem}
\begin{proof}
    Our proof builds upon the facts that for arbitrary $a,b \in \Ibb$ we have
    \begin{align}
        \vert a - b\vert = \int_\Ibb \vert \ONE_{[0,a)} - \ONE_{[0,b)} \vert \de \lambda
        = \int_\Ibb \vert \ONE_{(a,1]} - \ONE_{(b,1]} \vert \de \lambda,
    \end{align}
    and that (by the definition of the quantile function) for every 
    $C \in \kc^d$, every $\xvec \in \Ibb^{d-1}$, every $\tau \in (0,1]$ and every 
    $v \in [0,1]$ the following equivalence holds: 
    (i) $Q_C^\tau(\xvec)>v$ if, and only if (ii) $\tau > K_C(\xvec,[0,v])$. \\
    Hence, setting 
    $$V:=\intl_\Ibb \norm{ Q_{C_1}^\tau - Q_{C_2}^\tau}_{A,1} \de \lambda(\tau)$$ 
    and using Fubini's theorem, it follows that
    \begin{align*}
          V &= \int_\Ibb \int_{\Ibb^{d-1}} \int_\Ibb \abs{\ONE_{[0,Q^\tau_{C_1}(\xvec))}(v) 
          - \ONE_{[0,Q^\tau_{C_2}(\xvec))}(v)} \de \lambda(v) \de \mu_A(\xvec) \de \lambda(\tau) \\
          &= \int_\Ibb \int_{\Ibb^{d-1}} \int_\Ibb \abs{\ONE_{(K_{C_1}(\xvec,[0,v]),1]} (\tau)
          - \ONE_{(K_{C_2}(\xvec,[0,v]),1]}(\tau)} \de \lambda(v) \de \mu_A(\xvec) \de \lambda(\tau) \\
          &= \int_\Ibb \int_{\Ibb^{d-1}} \int_\Ibb \abs{\ONE_{(K_{C_1}(\xvec,[0,v]),1]} (\tau)
          - \ONE_{(K_{C_2}(\xvec,[0,v]),1]}(\tau)}  \de \lambda(\tau) \de \mu_A(\xvec) \de \lambda(v) \\
          &= \int_\Ibb \int_{ \Ibb^{d-1} } \abs{ K_{C_1}(\xvec, [0, y]) - K_{C_2}(\xvec, [0, y]) }\de \mu_A(\xvec) \de \lambda(v) \\
          &=D_{A,1}(C_1,C_2)
    \end{align*}
    Having this, using Lemma \ref{Lem2025063002} with $p=1$ completes the proof.
\end{proof}

\begin{ex}[cube copula cont.]
    Again consider the three-dimensional copula $C^{\textrm{cube}} \in \kc^3_\Pi$ 
    from Example \ref{ex:cube}. 
    In this case $K_{C^{\textrm{cube}}}(\xvec,\cdot)$ 
     is either the uniform distribution on $I_{2,1}$ or on $I_{2,2}$ and it is straightforward
     to verify that the $\tau$-quantile function 
     $Q_{ C^{\textrm{cube}} }^\tau$ is given by
    \begin{align*}
        Q_{ C^{\textrm{cube}} }^\tau(\xvec) = \tfrac{\tau}{2} \ONE_{ I_{2, 1}^2 \cup I_{2,2}^2 }(\xvec) + \tfrac{\tau+1}{2} \ONE_{(I_{2, 2} \times I_{2, 1}) \cup (I_{2, 1} \times I_{2, 2})}(\xvec).
    \end{align*}
    Since the partial vine copula $\psi(C^{\textrm{cube}})$ coincides with the independence 
    co\-pula $\Pi$, we have
    \begin{align*}
        \norm{ Q_{ C^{\textrm{cube}} }^\tau - Q_{ \psi(C^{\textrm{cube}}) }^\tau }_1 = \tfrac{1}{2} \abs{ \tfrac{\tau}{2} - \tau } + \tfrac{1}{2} \abs{ \tfrac{\tau+1}{2} - \tau } = \tfrac{1}{4},
    \end{align*}
    for every $\tau \in (0,1]$, so, having in mind Theorem \ref{thm:average.quantile} 
    also from the perspective of quantile regression
    the approximation quality of the partial vine may be very poor.  
\end{ex}

\section{Estimation in the bivariate setting}
We conclude this paper with some results on estimating the mean and the quantile regression functions via the empirical checkerboard estimator in dimension $d=2$. 
Suppose that $C \in \kc^2$ is fixed and that $(U_1,V_1)\ldots,(U_n,V_n)$ is a sample from 
$(U,V)\sim C$. As before let $E_n$ denote the induced empi\-rical copula and 
$\mathfrak{Cb}_{N(n)}(E_n)$ its checkerboard approximation with $N(n)=\lfloor n^s\rfloor$,
for some fixed $s \in (0,\frac{1}{2})$. Then, as shown in \cite{ALFT} (also see 
\cite{FuchsTrutschnig2020,Griessenberger2022,JGT}), the sequence $(\mathfrak{Cb}_{N(n)}(E_n))_{n \in \mathbb{N}}$ converges weakly conditional to $C$ with proba\-bility $1$, without any 
regularity/smoothness restrictions on $C$. 
In other words: With probability $1$, for $\lambda$-almost every
$x \in \Ibb$, the conditional distributions $K_{\mathfrak{Cb}_{N(n)}(E_n)}(x,\cdot)$ 
converge weakly to $K_C(x,\cdot)$ for $n \rightarrow \infty$. 
Although weak conditional convergence has been proved in full generality in the 
afore-mentioned papers, to the best of our knowledge, asymptotics of the 
checkerboard estimator $\mathfrak{Cb}_{N(n)}(E_n)$ in its pure (unaggregated) as 
well as its aggregated form are still unknown. \\
Resturning to consistency in the regression context, the following result is 
a direct consequence of weak condition convergence.  
\begin{thm} \label{thm:consistence.mean}
Suppose that $C \in \kc$ and that $(U_1,V_1)\ldots,(U_n,V_n)$ is a sample from 
$(U,V)\sim C$ with empirical copula $E_n$. Furthermore set $N(n):=\lfloor n^s\rfloor$
for some fixed $s \in (0,\frac{1}{2})$.
Then for $\lambda$-almost every $x \in \Ibb$ we have 
$$
\lim_{n \rightarrow \infty} r_{\mathfrak{Cb}_{N(n)}(E_n)}(x)=r_C(x),
$$
so in particular 
$$
\lim_{n \rightarrow \infty} \Vert r_{\mathfrak{Cb}_{N(n)}(E_n)}-r_C \Vert_p=0
$$
holds for every $p \in [1,\infty)$.
\end{thm}

Since weak convergence of a sequence $(F_n)_{n \in \mathbb{N}}$ 
of distribution functions to a distribution function $F$ 
is equi\-valent to pointwise convergence 
of the corresponding quasi-inverses $(F_n^-)_{n \in \mathbb{N}}$ in every continuity point
of $F^-$ (see, e.g., \cite{VDW}), the afore-mentioned property on 
weak conditional convergence implies the following: 
There exists some $\Lambda \in \mathcal{B}(\Ibb)$ with $\lambda(\Lambda)=1$,
such that for every $x \in \Lambda$ and every continuity point $\tau \in \Ibb$ of 
$\tau \mapsto Q_C^\tau(x)$ we have
$$
\lim_{n \rightarrow \infty} Q_{\mathfrak{Cb}_{N(n)}(E_n)}^\tau(x) = Q_C^\tau(x).
$$
For proving our second main result of this section - consistency of the empirical 
checkerboard estimator for quantile function - 
we will use the following technical lemma, in which the set
$S_q$ for $q \in (0,1)$ is defined by
\begin{equation}\label{eq:plateu}
S_q := \{x \in \Ibb: \, q \textrm{ is a discontinuity point of } \tau \mapsto Q_C^\tau(x) \} \in \mathcal{B}(\Ibb).
\end{equation}
\begin{lem} \label{Lem20250923}
There are at most countably many $q \in (0,1)$ with $\lambda(S_q)>0$. 
\end{lem} 
\begin{proof}
Suppose that $q \in (0,1)$ is a discontinuity point of $\tau \mapsto Q_C^\tau(x)$. Then, left-continuity implies that $Q_C^q(x) < Q_C^{q+}(x)$, where $Q_C^{q+}(x)$ denotes the 
right-hand limit of $\tau \mapsto Q_C^{\tau}(x)$ at $q$. Hence, by definition of the quantile function,  
there exists some $\Delta>0$, such that $y \mapsto K_C(x, [0, y])$ is constant on the interval 
$[Q_C^q(x),Q_C^q(x) + \Delta]$. In case $S_q$ fulfills $\lambda(S_q)>0$, the previous observation directly implies that
$$
\lambda_2 \left(\{(x,y) \in \Ibb^2: \, K_C(x,[0,y]) = q \} \right) >0.
$$
Setting $\Psi_C(x,y):=K_C(x,[0,y])$ and letting $[\Psi_C]_z$ denote the upper 
$z$-level of $\Psi_C$, we obviously have that the function $\ell: \Ibb \rightarrow \Ibb$, 
defined by
$$
\ell(z)= \lambda_2([\Psi_C]_z),
$$
is non-increasing on $\Ibb$. Every $q$ fulfilling $\lambda(S_q)>0$ obviously is a 
discontinuity point of $\ell$. As non-increasing function, $\ell$
can have at most countably many discontinuity points, and the proof is complete.
\end{proof}

As an immediate consequence of Lemma \ref{Lem20250923}, the following statements hold.

\begin{thm} \label{Theo20250919}
Suppose that $C \in \kc$, that $(U_1,V_1)\ldots,(U_n,V_n)$ is a sample from 
$(U,V)\sim C$, and that $E_n$ is the empirical copula. Furthermore set $N(n):=\lfloor n^s\rfloor$,
for some fixed $s \in (0,\frac{1}{2})$.
Then, for all but at most countably many
$\tau \in (0,1)$ and $\lambda$-almost every $x$, we have 
$$
\lim_{n \rightarrow \infty} Q_{\mathfrak{Cb}_{N(n)}(E_n)}^\tau(x) = Q_C^\tau(x).
$$
In particular, for all but at most countably many $\tau \in (0,1)$ and every $p \in [1,\infty)$, 
it holds that
$$
\lim_{n \rightarrow \infty} \Vert Q_{\mathfrak{Cb}_{N(n)}(E_n)}^\tau- Q_C^\tau \Vert_p=0.
$$
\end{thm}

We conclude our discussion with two concrete examples - a Marshall-Olkin copula as well as
a Clayton copula - and a small simulation study illustrating the speed of convergence.   

\begin{ex}[Marshall-Olkin] \label{ExMO}
It is well-known (see \cite[$\S$6.4]{DS2016} that the Marshall-Olkin copula $M_{\alpha,\beta} \in \kc^2$, for $\alpha, \beta > 0$ and $a(x) := x^\frac{\alpha}{\beta}$, is given by
\begin{align*}
    M_{\alpha, \beta}(x,y) := x^{1-\alpha} y \ONE_{ [0, a(x)] }(y) + x y^{1-\beta} \ONE_{ (a(x), 1] }(y) \hspace{1cm} ((x, y) \in \Ibb^2).
\end{align*}
As shown in \cite{Trutschnig2017}, the associated Markov kernel for $x \in (0,1)$ 
is given by
\begin{align*}
    K_{ M_{\alpha, \beta} }(x, [0, y]) = (1-\alpha) x^{-\alpha} y \ONE_{ [0, a(x)) }(y) + y^{1-\beta} \ONE_{ (a(x), 1 ] }(y),
\end{align*}
implying that $K_{M_{\alpha,\beta}}(x,\cdot)$ has a point mass at $y = a(x)$. 
In fact, we have $y^- < y^+$ with
\begin{align*}
    y^-&:=K_{ M_{\alpha, \beta} }(x, [0, a(x)-]) = (1-\alpha) x^{\alpha(\frac{1}{\beta}-1)} \\
    y^+&=K_{ M_{\alpha, \beta} }(x, [0, a(x)])=x^{\alpha(\frac{1}{\beta}-1)}
\end{align*}
As a direct consequence the quantile function is of the form
\begin{align*}
    Q_{ M_{\alpha, \beta} }^\tau(x) = \tfrac{ \tau x^\alpha }{1- \alpha} \ONE_{[0, y^-)}(\tau) + a(x) \ONE_{[y^-, y^+]}(\tau) + \tau^\frac{1}{1-\beta} \ONE_{(y^+, 1] }(\tau). 
\end{align*}
Moreover, using eq. (\ref{2025051303}) it is striaghtforard to verify that the 
regression function is given by
\begin{align*}
    r_{ M_{\alpha, \beta} }(x) = 1 - \tfrac{1-\alpha}{2} x^{\alpha(\frac{2}{\beta}-1)} - 
    \tfrac{1 - x^{\alpha(\frac{2}{\beta}-1)}}{2-\beta} \hspace{1cm} (x \in \Ibb).
\end{align*}
Notice that for $M_{\alpha,\beta}$ there is no $q \in (0,1)$ fulfilling $\lambda(S_q)>0$,
with $S_q$ according to eq. (\ref{eq:plateu}).
.
\end{ex}

As second example we consider a member of the Archimedean family. 
In this case the regression function then does not admit an elementary analytic form. 

\begin{ex}[Clayton] \label{ExC}
The Clayton copula (see \cite[Example 2.1.5]{Hoferetal2018}), with $\theta > 0$, is defined by
\begin{align*}
    C_\theta(x,y) = (x^{-\theta} + y^{-\theta} - 1)^{-\frac{1}{\theta}} \hspace{1cm} ((x,y) \in \Ibb^2).
\end{align*}
By direct computation, one thus verifies that
\begin{align*}
    K_{C_\theta}(x, [0, y]) = x^{-\theta-1} (x^{-\theta}+y^{-\theta}-1)^{-\frac{1}{\theta}-1} \hspace{1cm} ((x, y) \in \Ibb^2).
\end{align*}
Obviously, the Markov kernel is a strictly increasing function of $y$, we have $K_{C_\theta}(x, [0, Q_{C_\theta}^\tau(x) ] ) = \tau$ as well as 
\begin{align*}
    Q_{C_\theta}^\tau(x) = \rb{1+x^{-\theta}(\tau^{-\frac{\theta}{\theta+1}}-1)}^{-\frac{1}{\theta}}.
\end{align*}
With this, using change of coordinates we easily obtain $r_{C_\theta}(x) = \int_0^1 Q_{C_\theta}^y(x) dy$ or
\begin{align*}
    r_{C_\theta}(x) = \intl_0^1 \rb{1+x^{-\theta}(y^{-\frac{\theta}{\theta+1}}-1)}^{-\frac{1}{\theta}} d\lambda(y).
\end{align*}
As in the previous example there is no $q \in (0,1)$ fulfilling $\lambda(S_q)>0$, 
so the empirical checkerboard estimator is strongly consistent for every 
quantile $\tau \in (0,1]$.
\end{ex}

\begin{figure}[ht!]
\centering
\resizebox*{0.7\textwidth}{!}{\includegraphics{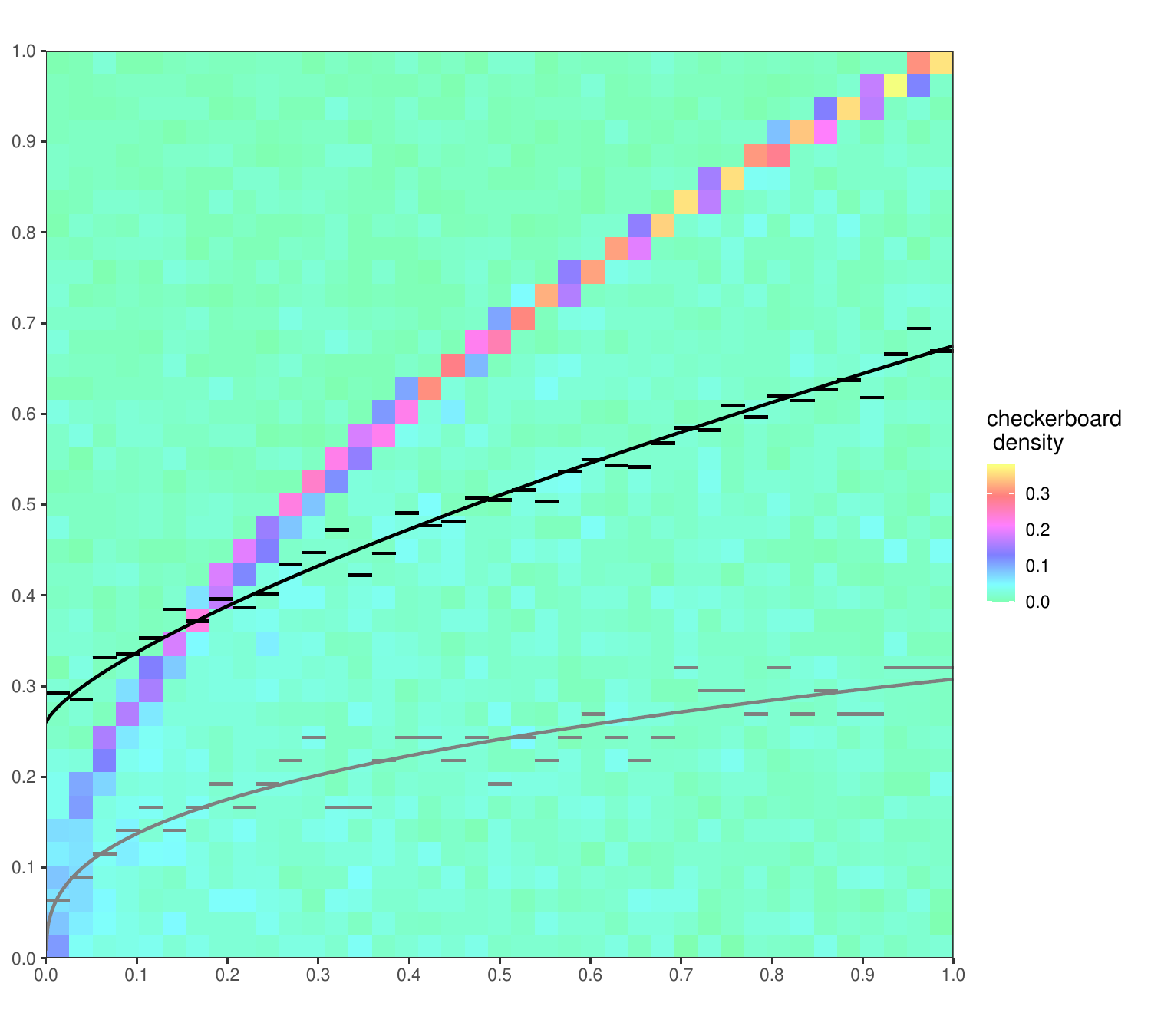}}
\caption{Empirical $N=63$ checkerboard density for a sample of size $n = 10.000$ from the
Marshall Olkin copula with parameters $(\alpha, \beta) = (0.35, 0.65)$, true mean and 
median regression functions (solid black and gray line, respectively), and
corresponding estimators $r_{\mathfrak{Cb}_N(E_n)}$ and $Q_{\mathfrak{Cb}_N(E_n)}^{0.2}$ 
(black and gray step functions).}
\label{MO_emp_cb_plot}
\end{figure}

We close this section with a small simulation study illustrating the estimation procedure 
and the speed of convergence of the involved, checkerboard-based estimators and consider the 
Marshall Olkin copula with parameters $(\alpha, \beta) = (0.35, 0.65)$ and a Clayton 
copula with $\theta = 2$. 
Figures \ref{MO_emp_cb_plot} and \ref{C_emp_cb_plot} depict the density of the 
empirical $N$-checkerboards for a sample of size $n = 10.000$ and resolution 
$N = \lfloor n^{0.4} \rfloor$. The black and the gray solid lines correspond 
to the true mean and quantile regression function, respectively, whose explicit 
expressions were derived in Examples \ref{ExMO} and \ref{ExC}. The black and gray 
step functions correspond to $r_{\mathfrak{Cb}_N(E_n)}$ and $Q_{\mathfrak{Cb}_N(E_n)}^{0.2}$, 
respectively.

In addition, Figure \ref{MO_C_boxplots} illustrates the speed of convergence of $r_{\mathfrak{Cb}_N(E_n)}$ and $Q_{\mathfrak{Cb}_N(E_n)}^{0.2}$. 
For each of the samples sizes $n$ mentioned on the $x$-axis 
we drew a sample of $n$ (from the considered copula), numerically calculated 
$$
\Vert r_{\mathfrak{Cb}_{N(n)}(E_n)}-r_C \Vert_1,\quad
 \Vert Q_{\mathfrak{Cb}_{N(n)}(E_n)}^\tau- Q_C^\tau \Vert_1,
$$
repeated the procedure $R=500$ times and summarized the obtained results as boxplots. 
All computations were performed in \texttt{R} using the packages \texttt{copula} and \texttt{qad}.

\vspace{1cm}
\noindent\emph{Acknowledgements} \\[3mm]
Both authors gratefully acknowledge the support of the WISS 2025 project `IDA-lab Salzburg' 
(20204-WISS/225/197-2019 and 20102-F1901166-KZP).

%In these examples, $N = N(n) = \lfloor n^\frac{4}{10} \rfloor$ and the $L^p$-distances were approximated by a uniform sample of size $100$. This step was repeated $500$ times, for each size of the shown copula samples. The boxplots of these observations in Figure \ref{MO_C_boxplots} illustrate the decaying character that we expected according to our theoretical results. 

\begin{figure}
\centering
\resizebox*{0.7\textwidth}{!}{\includegraphics{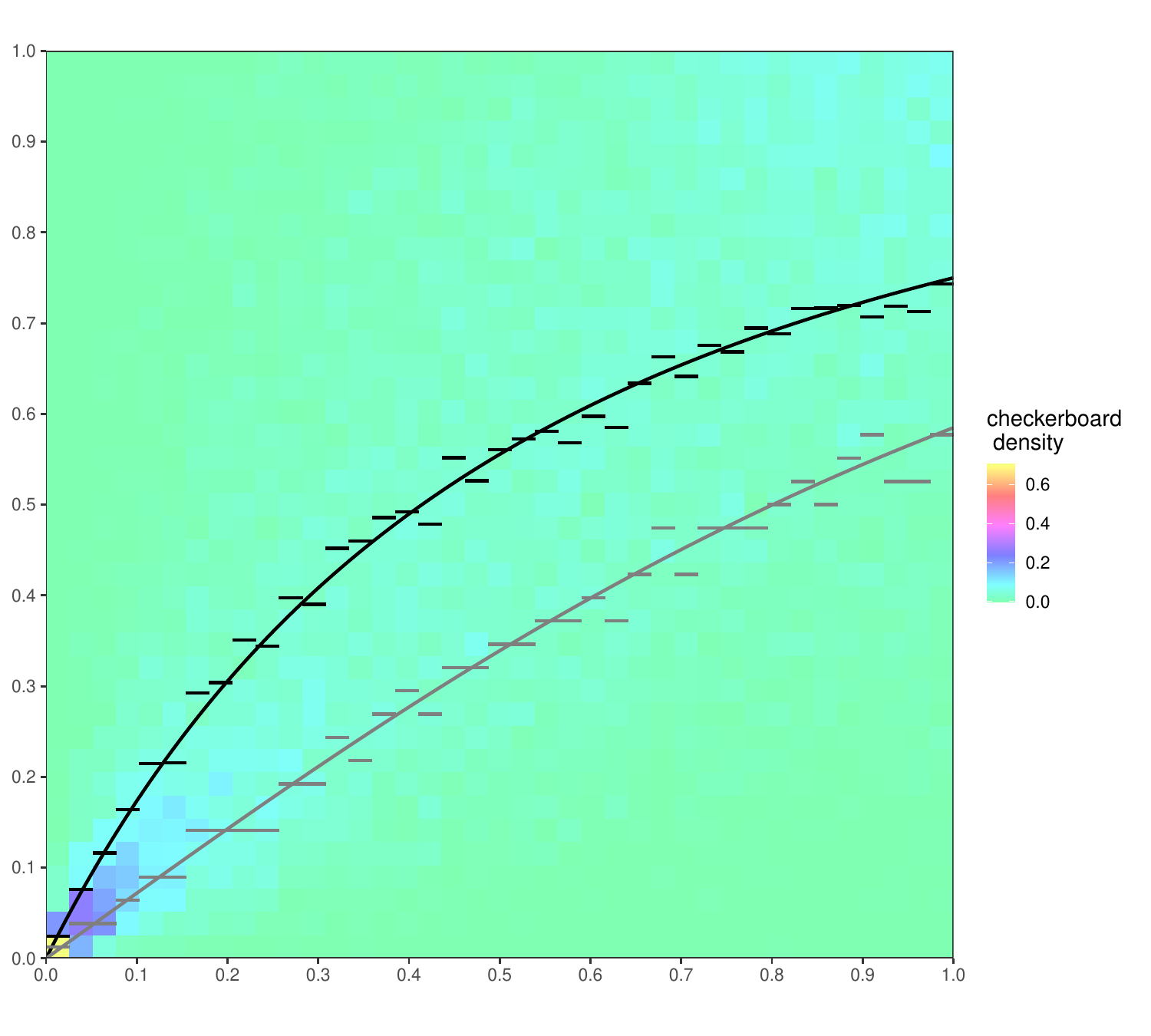}}
\caption{Empirical $N=63$ checkerboard density for a sample of size $n = 10.000$ from the
Clayton copula with parameter $\theta=2$, true mean and 
median regression functions (solid black and gray line, respectively), and the
corresponding estimators $r_{\mathfrak{Cb}_N(E_n)}$ and $Q_{\mathfrak{Cb}_N(E_n)}^{0.2}$ 
(black and gray step functions).}
\label{C_emp_cb_plot}
\end{figure}

\begin{figure}[ht!]
\centering
\resizebox*{0.8\textwidth}{!}{\includegraphics{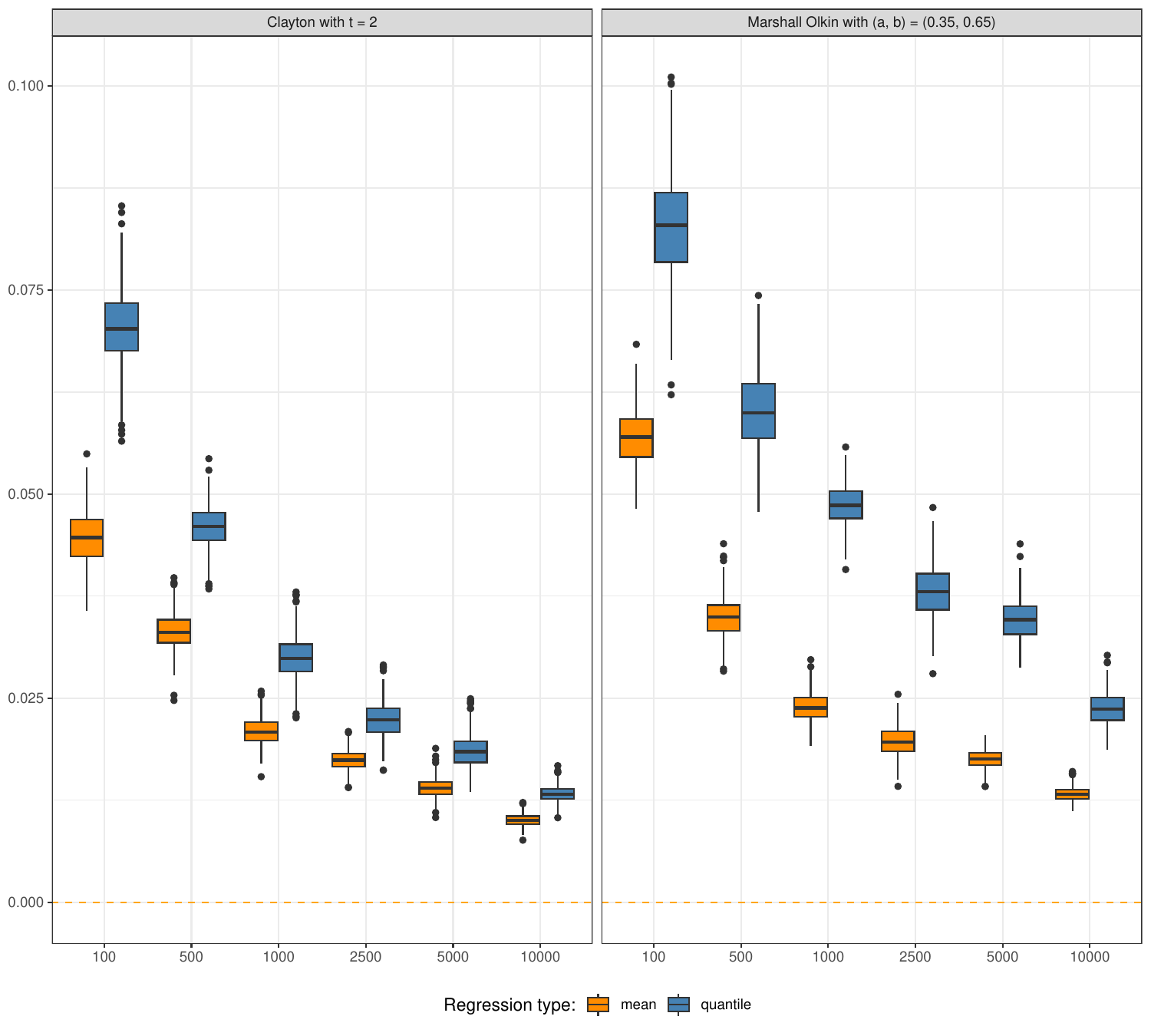}}
\caption{Boxplot summarizing the $L^1$-distances between the estimated and the true mean regression as well 
as the estimated and the true median regression function, respectively. 
For each sample size $n$ (on the $x$-axis) a total of $R=500$ runs were performed. }
\label{MO_C_boxplots}
\end{figure}

%\clearpage
\appendix

\section{Proofs to Section \ref{SecNotPre}} \label{AppProNotPrel}

\begin{proof}[Proof of Lemma \ref{Lem2025062701}]
For brevity, we write
\begin{align*}
Q_d(\xvec_{1:d-2}, B) := \intl_{ \Ibb } K_C( \xvec_{1:d-1}, B ) K_{C_{1:(d-1)}}(\xvec_{1:d-2}, \de x_{d-1}) .
\end{align*}
Clearly, for fixed $\xvec_{1:d-2} \in \Ibb^{d-2}$, the assignment $B \mapsto Q_d(\xvec_{1:d-2}, B)$ fulfills 
the properties of a probability measure. 
Furthermore, by \cite[Lemma 14.23]{Klenke2020}, for fixed $B \in \mathcal{B}(\Ibb)$ the mapping
$\xvec_{1:d-2} \mapsto Q_d(\xvec_{1:d-2}, B)$ is Borel measurable. In other words: 
$Q_d: \Ibb^{d-2} \times \mathcal{B}(\Ibb) \rightarrow \Ibb$ is a Markov kernel and it remains to show that 
it is a Markov kernel of $C_{1:(d-2),d}$. For $E_1,E_2,\ldots,E_{d-2},E_d \in \mathcal{B}(\Ibb)$, 
setting $\mathbf{E}=E_1\times E_2 \times \ldots \times E_{d-2} \times \Ibb \times E_d$ and 
using disintegration twice (first for $\mu_C$, then for $\mu_{C_{1:(d-1)}}$) yields
\begin{align*}
    \mu_C(\mathbf{E}) &= \int_{E_1\times E_2 \times \ldots \times E_{d-2} \times \Ibb} 
    K_C(\mathbf{x}_{1:d-1},E_d)\, \de \mu_{C_{1:d-1}}(\mathbf{x}_{1:d-1}) \\
    &=  \int_{E_1\times E_2 \times \ldots \times E_{d-2}} \int_\Ibb  
    K_C(\mathbf{x}_{1:d-1},E_d)\, K_{C_{1:d-1}}(\mathbf{x}_{1:d-2},\de x_{d-1})\, \de \mu_{C_{1:d-2}}(\mathbf{x}_{1:d-2}) \\
    &= \int_{E_1\times E_2 \times \ldots \times E_{d-2}}  Q_d(\xvec_{1:d-2}, E_d) \de \mu_{C_{1:d-2}}(\mathbf{x}_{1:d-2}) 
\end{align*}
Considering $\mu_C(\mathbf{E})= \mu_{C_{1:(d-2),d}}(E_1 \times \ldots \times E_{d-2} \times E_d)$
and using the fact that the family of all rectangles of the form $E_1 \times \ldots \times E_{d-2} \times E_d$
constitute a semiring generating $\mathcal{B}(\Ibb^{d-1})$ this completes the proof.
\end{proof}

\begin{proof}[Proof of Lemma \ref{Lem20250922}]
For every $v \in \Ibb$ and $n \in \N$ we obviously have
\begin{align*}
\{\xvec \in \Ibb^{d-1}: f(\xvec)\geq v + \tfrac{1}{n}\} &\subseteq \{\xvec \in \Ibb^{d-1}: f(\xvec)> v \} 
  \subseteq \{\xvec \in \Ibb^{d-1}: f(\xvec)\geq  v \} \\
  &\subseteq \{\xvec \in \Ibb^{d-1}: f(\xvec)\geq  v -\tfrac{1}{n} \},
\end{align*}
which directly yields 
$$
\overline{m}_{f,C}(v+\tfrac{1}{n}) \leq m_{f,C}(v) \leq \overline{m}_{f,C}(v) \leq \overline{m}_{f,C}(v-\tfrac{1}{n}).
$$
Considering $n \rightarrow \infty$ and using monotonicity of all involved functions we get
$$
\overline{m}_{f,C}(v+) \leq m_{f,C}(v) \leq \overline{m}_{f,C}(v) \leq \overline{m}_{f,C}(v-).
$$
As a direct consequence, we can only have $m_{f,C}(v) \neq \overline{m}_{f,C}(v)$ if $v$ is a discontinuity point 
of $\overline{m}_{f,C}$. By monotonicity of $\overline{m}_{f_C}$, however, the set of discontinuity points 
of $\overline{m}_{f_C}$ is at most countably infinite, i.e., 
$$
m_{f,C}(v) =  \overline{m}_{f,C}(v)
$$ 
holds outside an at most countably infinite set. In particular, it follows that 
$V := \curbr{ v \in \Ibb : m_{f,C}(v) = \overline m_{f,C}(v) }$ is dense in $\Ibb$. Having 
this, it is straightforward to show that $f_{C, \downarrow}$ and $\overline f_{C, \downarrow}$ coincide on $\Ibb$. 
%Suppose now the existence of $u \in \Ibb$, with $\overline f_{C, \downarrow}(u) > f_{C, \downarrow}(u)$. Then, we can find $v_0 > f_{C, \downarrow}(u)$, with $\overline m_{ f_{C, \downarrow} }(v_0) > u$. In addition, due to the denseness of $V$, we can find $w \in V$, with $v_0 \geq w > f_{C, \downarrow}(u)$. From $w > f_{C, \downarrow}(u)$, we infer that $m_{ f_{C, \downarrow} }(w) \leq u$ and thereby $\overline m_{f_C}(w) = m_{f_C}(w) \leq u$. Finally, due to $\overline m_{ f_{C, \downarrow} }$ being non-increasing, $\overline m_{f_C}(v_0) \leq u$, inflicting a contradiction to the fact that $m_{f_C}(v_0) > u$. In a similar fashion, one shows that $\overline f_{C, \downarrow}(u) < f_{C, \downarrow}(u)$ is impossible, so that we may conclude that indeed $f_{C, \downarrow}(u) = \overline f_{C, \downarrow}(u)$, for all $u \in \Ibb$. The identity for general measurable $f : \Ibb \rightarrow \Ibb$ follows in an analogous way.
\end{proof}

\clearpage

\bibliographystyle{apalike}
%\bibliography{References}{}

\end{document}